\documentclass[a4paper]{amsart}

\usepackage[lmargin=1in,rmargin=1in,tmargin=1in,bmargin=1in]{geometry}
\usepackage{amsthm,amsfonts,amssymb,amsmath}
\usepackage{graphicx}
\usepackage{tikz}
\usepackage{tikz-cd}
\usepackage{mathtools,leftindex}
\usetikzlibrary{positioning,arrows}
\usepackage{enumitem}
\usepackage{subcaption}
\usepackage[colorlinks=true, linkcolor=blue, citecolor=blue]{hyperref}

\usetikzlibrary{calc}
\newcommand*{\vertchar}[2][0pt]{%
  \tikz[
    inner sep=0pt,
    shorten >=-.15ex,
    shorten <=-.15ex,
    line cap=round,
    baseline=(c.base),
  ]\draw [line width=0.05mm,  black]
    (0,0) node (c) {#2}
    ($(c.south)+(#1,0)$) -- ($(c.north)+(#1,0)$);%
}

\newcommand{\bZ}{\mathbb{Z}}

\newcommand{\cT}{\mathcal{T}}
\newcommand{\cE}{\mathcal{E}}
\newcommand{\cL}{\mathcal{L}}

\newcommand{\tA}{\tilde{A}}
\newcommand{\tB}{\tilde{B}}
\newcommand{\tC}{\tilde{C}}

\newcommand{\tX}{\tilde{X}}
\newcommand{\tY}{\tilde{Y}}

\newcommand{\tv}{\tilde{v}}
\newcommand{\te}{\tilde{e}}
\newcommand{\tf}{\tilde{f}}

\newcommand{\hA}{\widehat{A}}
\newcommand{\hX}{\widehat{X}}
\newcommand{\hY}{\widehat{Y}}
\newcommand{\hgamma}{\widehat{\gamma}}

\newcommand{\bA}{\overline{A}}
\newcommand{\bX}{\overline{X}}

\newcommand{\brho}{\overline{\rho}}
\newcommand{\bH}{\overline{H}}
\newcommand{\bM}{\overline{M}}
\newcommand{\bN}{\overline{N}}
\newcommand{\bD}{\overline{D}}
\newcommand{\bK}{\overline{K}}
\newcommand{\bd}{\overline{d}}
\newcommand{\bm}{\overline{m}}
\newcommand{\bS}{\overline{S}}

\newcommand{\RNum}[1]{\uppercase\expandafter{\romannumeral #1\relax}}

\newcommand{\CAT}{\operatorname{CAT}}
\newcommand{\diam}{\operatorname{diam}}

\newcommand{\Hull}{\operatorname{Hull}}
\newcommand{\Aut}{\operatorname{Aut}}

\newcommand{\lk}{\operatorname{lk}}
\newcommand{\pf}{\pitchfork}
\newcommand{\la}{\langle}
\newcommand{\ra}{\rangle}
\newcommand{\lv}{\lvert}
\newcommand{\rv}{\rvert}

\newcommand{\llf}{\left\lfloor}
\newcommand{\rrf}{\right\rfloor}

\newcommand{\pstab}{\operatorname{\textup{\vertchar{S}tab}}}
\newcommand{\stab}{\operatorname{Stab}}

\newtheorem{thm}{Theorem}[section]
\newtheorem{prop}[thm]{Proposition}
\newtheorem{cor}[thm]{Corollary}
\newtheorem{lem}[thm]{Lemma}
\newtheorem{exam}[thm]{Example}
\newtheorem{conj}[thm]{Conjecture}
\theoremstyle{definition}
\newtheorem{defn}[thm]{Definition}
\newtheorem{rmk}[thm]{Remark}

\numberwithin{equation}{section}

\begin{document}

\title{Finite Stature in Graphs of Cube Complexes with Cyclonormal Edges}
\author{Changqian Li}

\address{Department of Mathematics, The Ohio State University, Columbus, OH, 43210, U.S.}
\email{li.13132@osu.edu}

\begin{abstract}
Given a compact cube complex $X$ that splits as a graph of virtually special cube complexes. Suppose that the fundamental groups of edge spaces are cyclonormal in the fundamental groups of adjacent vertex spaces. We show that $\pi_1X$ has finite stature with respect to vertex groups in the sense of Huang--Wise. In particular, when vertex groups are hyperbolic, this allows us to deduce virtual specialness of $X$.
\end{abstract}

\maketitle

\section{Introduction}
The notion of virtually special cube complexes was introduced by Haglund and Wise in \cite{Haglund-Wise2008}. Building on this framework, Haglund and Wise proved the Malnormal Combination Theorem below.
\begin{thm}[{\cite[Theorem 8.2]{haglund-wise2012}}]\label{malnormal combination}
Let $X$ be a compact connected nonpositively cured cube complex, and let $P$ be a hyperplane in $X$ such that the following hold:
\begin{enumerate}
    \item $\pi_1X$ is hyperbolic.
    \item $P$ is an embedded $2$-sided hyperplane in $X$.
    \item $\pi_1P$ is malnormal in $\pi_1X$.
    \item Each component of $X\setminus N_o(P)$ is virtually special.
\end{enumerate}
Then $X$ is virtually special.
\end{thm}
This theorem plays an essential role in the proofs of Wise's Malnormal Quasiconvex Hierarchy Theorem \cite{AgolGrovesManning2016} and Malnormal Special Quotient Theorem \cite{wise2021}. 
These results are also key ingredients in Agol’s proof of the long-standing Virtual Haken Conjecture \cite{Agol2013VirtualHaken}.
However, the above theorems are restricted to the hyperbolic setting. Related generalizations in the relatively hyperbolic context have been studied in \cite{wise2021}, \cite{GrovesManning2022}, and \cite{OregonReyes2023}. 
Given the many interesting consequences of virtual specialness, it is desirable to extend this theory to groups that are not necessarily (relatively) hyperbolic, which remains largely unexplored.

Recall that a subgroup $H$ of a group $G$ is \textit{cyclonormal} if for any $g\in G\setminus H$, the intersection $H\cap gHg^{-1}$ is either trivial or a cyclic subgroup.
We review the notion of \textit{graph of cube complexes} in Section \ref{section: graph of cube complexes}. Suppose a cube complex $X$ splits as a graph of nonpositively curved cube complexes. We say that it has \textit{cyclonormal edges} if the fundamental group of each edge space is cyclonormal in fundamental groups of adjacent vertex spaces (see Definition \ref{definition: almost cyclonormal edges}).
Our main result is the following:
\begin{thm}[cf. Corollary~\ref{corollary: virtually special}]\label{main result}
Let $X$ be a compact cube complex that splits as a graph of nonpositively curved cube complexes with cyclonormal edges. If the vertex groups are hyperbolic, then $X$ is virtually special.
\end{thm}
Notice that the main difference between this result and Theorem~\ref{malnormal combination} is that, although the vertex groups are required to be hyperbolic, $\pi_1 X$ itself need not be hyperbolic or relatively hyperbolic. 
The action of $\pi_1X$ on the associated Bass--Serre tree is not necessarily acylindrical. Moreover, we caution the reader that, despite the assumption of cyclonormal edges, the pointwise stabilizers of infinite subtrees in the Bass--Serre tree can be more complicated than cyclic groups.

Furthermore, Theorem~\ref{main result} generalizes the malnormality assumption in Theorem~\ref{malnormal combination} to cyclonormality.
Note that Theorem~\ref{main result} fails if we drop the assumption of cyclonormal edges.
For example, Wise constructed a graph of graphs in \cite{Wise1996NPCSquareComplexes} that is not virtually special.

One motivation for studying virtual specialness outside the hyperbolic setting comes from the classification of lattices in certain locally compact groups. For example, let $\Delta$ be a right-angled building realized as a graph product of finite groups. Then uniform lattices in $\Aut(\Delta)$ are commensurable provided that they are virtually special \cite{Shephard2024rightangledbuilding}. Moreover, virtual specialness is closely related to the quasi-isometry rigidity of certain cubical groups \cite{Huang2018Commensurability}. 

While Theorem~\ref{main result} generalizes Theorem~\ref{malnormal combination}, its proof is based on \cite[Theorem~1.4]{Huang--Wise2024fintiestature}, whose argument relies on the Malnormal Special Quotient Theorem and, in particular, on Theorem~\ref{malnormal combination}. Consequently, Theorem~\ref{main result} does not provide a new proof of Theorem~\ref{malnormal combination}.

In fact, Theorem~\ref{main result} follows as a corollary of Theorem~\ref{main theorem} via the notion of finite stature.

A group $G$ has \textit{finite stature} with respect to a collection $\Omega = \{G_\lambda\}_{\lambda\in \Lambda}$ of subgroups if for each $H\in \Omega$, there are finitely many $H$-conjugacy classes of infinite subgroups of the form $H\cap C$, where $C$ is an intersection of $G$-conjugates of elements of $\Omega$.
Finite stature was introduced by Huang and Wise in~\cite{Huang--Wise2019statureseparabilitygraphsgroups}. 
In the same paper, when a group $G$ splits as a graph of groups, finite stature is reformulated in terms of pointwise stabilizers for the action of $G$ on the Bass--Serre tree. 
In particular, when a compact cube complex $X$ splits as a graph of nonpositively curved cube complexes with hyperbolic vertex groups, the same authors showed that $X$ is virtually special if and only if $\pi_1X$ has finite stature with respect to vertex groups \cite{Huang--Wise2024fintiestature}.
Theorem \ref{main result} is a corollary of the following:
\begin{thm}\label{main theorem}
Let $X$ be a compact cube complex that splits as a graph of virtually special cube complexes with cyclonormal edges, then $\pi_1X$ has finite stature with respect to the vertex groups. 
\end{thm}
Compared with Theorem~\ref{main result}, this result only assumes that the vertex spaces are virtually special and does not require hyperbolicity. In fact, it was conjectured in \cite{Huang--Wise2024fintiestature} that the equivalence between finite stature and virtual specialness continues to hold when one assumes only that the vertex spaces are virtually special. In our setting, we therefore expect that Theorem~\ref{main result} remains valid under this weaker assumption.
\begin{conj}
Let $X$ be a compact cube complex that splits as a graph of virtually special cube complexes with cyclonormal edges, then $X$ is virtually special. 
\end{conj}

Theorem~\ref{main theorem} is a special case of more general results in Theorem~\ref{theorem: main theorem} and Corollary~\ref{corollary: general version of main thm} concerning almost cyclonormality.

When $X$ splits as a graph of graphs, Theorem~\ref{main theorem} can be deduced from Wise's argument in \cite[Part \RNum{1}, Theorem 4.5.4]{Wise1996NPCSquareComplexes} for cyclonormal VH-complexes.

\subsection*{Structure of the paper}
In section \ref{section: priliminary}, we provide some preliminaries about nonpositively curved cube complexes and virtually special cube complexes. 
Section \ref{section: projection} discusses the projection map in $\CAT(0)$ cube complexes.
In section \ref{section: stature}, we review the notation of finite stature, and discuss its relation with pointwise stabilizers on the Bass--Serre tree.
In Section \ref{section: quasiline}, we introduce the notation of quasilines and study hyperplanes of them. 
Finally, we prove Theorem \ref{main theorem} in Section \ref{section: main theorem}.

\subsection*{Acknowledgments}
The author would like to thank Jingyin Huang for his guidance on the topic and for many fruitful discussions.

\section{Preliminary} \label{section: priliminary}
In this section, we review definitions related to nonpositively curved cube complexes, hyperplanes, and special cube complexes. For example, see \cite{wise2021} for further details.
\subsection{Nonpositively curved cube complexes and hyperplanes}
\begin{defn}
An \textit{$n$-dimensional cube} is a copy of $I^n$ where $I = [0,1]$. Its subcubes are obtained by restricting some coordinates to $0$ and some to $1$. A \textit{cube complex} is obtained by gluing a collection of cubes along their subcubes isometrically. A \textit{flag complex} is a simplicial complex with the property that every finite set of pairwise adjacent vertices spans a simplex. A cube complex $X$ is \textit{nonpositively curved} if the link $\lk(v)$ of each $0$-cube (vertex) $v$ of $X$ is a flag complex. A nonpositively curved cube complex is \textit{$\CAT(0)$} if it is simply-connected.
\end{defn}

\begin{defn}
A \textit{midcube} in an $n$-cube $I^n$ is the subspace obtained by restricting one coordinate to $\frac{1}{2}$. A hyperplane $H$ in a $\CAT(0)$ cube complex $\tX$ is a connected subspace such that for each cube $C$ of $\tX$, either $H\cap C = \emptyset$ or $H\cap C$ consists of a midcube of $C$.
The \textit{carrier} of a hyperplane $H$ is the subcomplex $N(H)$ consisting of all closed cubes intersecting $H$. The \textit{open carrier} of $H$ is the subcomplex $N_o(H)$ consisting of open cubes intersecting $H$. 
An \textit{immersed hyperplane} in a nonpositively curved cube complex $X$ is the image of a map $H/\stab_{\pi_1X}(H) \to X$ where $H$ is a hyperplane in the universal cover $\tX$. 
\end{defn}
A hyperplane $H$ of a $\CAT(0)$ cube complex $\tX$ separates $\tX$ into two connected components \cite{Sageev1995}. The two components of $\tX\setminus N_o(H)$ are called \textit{halfspaces} of $H$, denoted by $H^+$, $H^-$. We use $H^\epsilon$ to denote a chosen halfspace of $H$, where $\epsilon\in \{+, -\}$.
\begin{defn}
Let $f:Y\to X$ be a combinatorial map of nonpositively curved cube complexes. We say $f$ is a \textit{local isometry} if for each $0$-cube $y$ of $Y$, the induced map of links $\lk(y)\to \lk(f(y))$ is an embedding and has full image (a subcopmlex $A$ of a simplicial complex $B$ is \textit{full} if each simplex of $B$ whose vertices are in $A$ is entirely contained in $A$).
\end{defn}

Suppose that $f: Y\to X$ is a local isometry of connected nonpositively curved cube complexes, then the map $\tilde{f}:\tY\to \tX$ between universal covers is an isometric embedding and $\tilde{f}(\tY)$ is a convex subcomplex of $\tX$. Moreover, $f$ is $\pi_1$-injective. See \cite[Lemma 2.11]{Haglund-Wise2008}. For a general discussion of complete nonpositively curved metric spaces, see \cite{BridsonHaefliger1999}. 
\begin{defn}
Let $f:Y\to X$ be a local isometry of connected nonpositively curved cube complexes. We say that $f$ \textit{represents} (the conjugacy class of) the subgroup $\pi_1Y \cong f_*(\pi_1Y)\leq \pi_1X$.
\end{defn}

\subsection{Combinatorial distance and translation length}
\begin{defn}
Let $p,q$ be two vertices of a $\CAT(0)$ cube complex $\tX$. A \textit{combinatorial path} of $\tX$ between $p$ and $q$ is a sequence $\gamma = (p = v_0,v_1,\dots, v_n = q)$ of vertices of $\tX$ such that for each $i = 1,\dots, n$, either $v_{i-1} = v_i$ or $v_{i-1}$, $v_i$ are endpoints of an edge of $\tX$. The \textit{length} of $\gamma$ is $n$. The \textit{combinatorial distance} $d(p,q)$ of two vertices $p,q$ is the smallest length $n$ of a combinatorial path between $p$ and $q$. A \textit{combinatorial geodesic} between $p,q$ is a combinatorial path between $p,q$ of length $d(p,q)$. 
A subcomplex $\tY\subset \tX$ is \textit{convex} if for any pair of vertices $p,q$, each combinatorial geodesic joining $p,q$ lies in $\tY$. The \textit{convex hull} $\Hull(S)$ of a subset $S\subset \tX$ is the smallest convex subcomplex of $\tX$ containing $S$. 
\end{defn}

We recall the following fact about combinatorial geodesics from \cite{Sageev1995}.
\begin{lem}
Let $\tX$ be a $\CAT(0)$ cube complex. 
A combinatorial path in $\tX$ is a geodesic if and only if it crosses distinct hyperplanes. The combinatorial distance between two vertices $p,q$ of $\tX$ is the number of hyperplanes separating $p$ and $q$.
\end{lem}
We review the following from \cite{Haglund2023}.
\begin{defn}
Let $\tX$ be a $\CAT(0)$ cube complex. An \textit{automorphism} $\phi$ of $\tX$ is an isometry of $\tX$ which sends every $n$-cube to an $n$-cube for each integer $n\geq 0$. 
The \textit{$\CAT(0)$ translation length} of $\phi$ is defined by $\delta(\phi) = \inf_{x\in \tX}d(x,\phi(x))$.
The \textit{combinatorial translation length} of $\phi$ is defined by $\lv \phi \rv = \inf_{v\in V(\tX)}d(v,\phi(v))$, where $V(\tX)$ denotes the set of vertices of $\tX$.
\end{defn}
\begin{rmk}\label{remark: elements are hyperbolic}
It follows from the definition that $|\phi|\geq \delta(\phi)$.
Let $X$ be a compact nonpositively curved cube complex. Then $\pi_1X$ acts properly and cocompactly by isometries on the universal cover $\tX$ of $X$, and every element of $\pi_1X$ is a semi-simple isometry (see \cite[\RNum{2}, Proposition 6.10]{BridsonHaefliger1999}). Because $\pi_1X$ is torsion-free, each nontrivial element $\phi\in \pi_1X$ is a hyperbolic isometry. In particular, $|\phi|\geq \delta(\phi)>0$.
\end{rmk}

\subsection{Special cube complexes}
We review the following from \cite{Haglund-Wise2008} and \cite{haglund-wise2012}.
\begin{defn}
A nonpositively curved cube complex $X$ is \textit{special} if every immersed hyperplane is embedded, 2-sided, and does not self-osculate, and no two hyperplanes inter-osculate.
It is \textit{virtually special} if it has a finite-sheeted covering space that is special.
A group is said to be \textit{special} if it is the fundamental group of a compact special cube complex.
\end{defn}

Let $f: A\to X$ be a local isometry of compact special cube complexes. The canonical completion and retraction construction, due to Haglund--Wise \cite[Proposition~6.5]{Haglund-Wise2008},
yields a finite-sheeted covering map $p: C(A,X)\to X$, together with an embedding $j: A\hookrightarrow C(A,X)$ and a cellular map $r: C(A,X)\to A$ such that $f = pj$ and $1_A = rj$. See also \cite[Definition~3.4]{haglund-wise2012} for a construction via the fiber product. It is also interpreted as a walker and imitator construction by  Shepherd \cite[Construction~3.1]{Shepherd2023}.

\subsection{Elevation}
We recall the following from \cite{haglund-wise2012}.
\begin{defn}
Let $f: Y\to X$ be a local isometry of connected nonpositively curved cube complexes and let $p:\hX\to X$ be a covering map. The fiber product $Y\otimes_X\hX$ is defined to be
\[
Y\otimes_X\hX = \{(a,b)\in Y\times \hX \mid f(a) = p(b)\}.
\]
An \textit{elevation} $\hY$ of $Y$ to $\hX$ is a connected component of $Y\otimes_X \hX$. In particular, we have the following commutative diagram.
\[
\begin{tikzcd}
\hY \arrow[r, "p_{\hX}"] \arrow[d, "p_Y"'] &
\widehat X \arrow[d, "p"] \\
Y \arrow[r, "f"'] &
X
\end{tikzcd}
\]
When $f:(Y,y)\to (X,x)$ and $\pi:(\hX,\hat{x})\to (X,x)$ are based maps with $f(y) = x = p(\hat{x})$, the \textit{based elevation} of $Y$ to $\hX$ is the component of $Y\otimes_{X}\hX$ containing the basepoint $(y,\hat{x})$.
\end{defn}
\begin{rmk}\label{remark: elevation and double cosets}
Identify $\pi_1Y$ and $\pi_1\hX$ with $f_*(\pi_1X)$ and $p_*(\pi_1\hX)$ respectively as subgroups of $\pi_1X$.
Then there is a one-to-one correspondence between elevations of $Y$ to $\hX$ and the double cosets $\pi_1\hX\backslash \pi_1X/ \pi_1Y$ (see \cite[Lemma~3.12]{ScottWall1979} or \cite[Lemma~2.6]{Wilton2008}). In particular, when $\hX$ is a regular cover of $X$, elevations correspond to cosets $\pi_1X/(\pi_1\hX\pi_1Y)$. 
Suppose that the based elevation $\hY$ embeds into $\hX$. Since $\hX$ is regular, all elevations embed into $\hX$. 
By identifying the based elevation with a subcomplex $\hY\subset \hX$, other elevations have the form $g\hY$ for $g\in \pi_1X$. Then under the action of $\pi_1X$ on $\hX$, the stabilizer of the elevation $g\tY$ is $\stab(g\hY) = \pi_1\hX\cdot g\pi_1Y{g^{-1}}$.
\end{rmk}

\section{Projection}\label{section: projection}
We recall the following from \cite[Appendix B]{Haglund-Wise2008}.
\begin{defn}
    Let $\tX$ be a $\CAT(0)$ cube complex and let $\tY\subset \tX$ be a convex subcomplex. The \textit{projection} from $\tX$ to $\tY$ is the combinatorial map $\Pi_{\tY}:\tX\to \tY$ sending each vertex $x\in\tX^0$ to the unique closest vertex in $\tY^0$.
\end{defn}

Let $X$ be a compact nonpositively curved cube complex. If subgroups $H,K\leq \pi_1X$ act cocompactly on convex subcomplexes $\tA,\tB$ of $\tX$ respectively, then $H\cap  K$ acts cocompactly on the convex subcomplex $\Pi_{\tA}\tB$ \cite[Lemma~2.8]{Fioravanti2023}. This indicates that the projection $\Pi_{\tA}\tB$ serves as an analogue of the intersection, even though $\tA\cap \tB$ may be empty. We thus denote $\Pi_{\tA}\tB$ by $\tA\pf \tB$ for brevity. Notice that this notation is not symmetric, since $\tA\pf \tB\subset \tA$, whereas $\tB\pf\tA\subset \tB$. 

The following theorem shows that there is an isometry $\tA \pf \tB \cong \tB \pf \tA$, which induces a bridge structure between convex subcomplexes. See \cite[Theorem 1.22]{Hagen2019} for a proof (see also \cite[Lemma 2.18]{ChatterjiFernosIozzi2016} and \cite[Theorem 5.2]{Shepherd2023}).
\begin{thm}[Bridge Theorem]\label{Thm: bridge theorem}
    Let $\tX$ be a $\CAT(0)$ cube complex and let $\tA,\tB\subset \tX$ be convex subcomplexes. Then the following hold:
    \begin{enumerate}
        \item A hyperplane $H$ crosses $\tA \pf \tB$ if and only if $H$ crosses both $\tA$ and $\tB$.
        A hyperplane $H$ separates $\tA \pf \tB, \tB\pf\tA$ if and only if $H$ separates $\tA,\tB$.
        \item The convex hull of $(\tA\pf\tB)\cup (\tB\pf\tA)$ is a convex subcomplex isomorphic to $C\times (\tA\pf\tB)$, where $C$ is the convex hull of a pair of vertices $a\in \tA\pf \tB$ and $b\in \tB\pf\tA$. In particular, $\tA\pf\tB$ and $\tB\pf\tA$ are isometric.
    \end{enumerate}
\end{thm}

\begin{cor}\label{corollary: property of projection}
Let $\tA,\tB,\tC$ be convex subcomplexes of a $\CAT(0)$ cube complex $\tX$, then 
\begin{enumerate}
    \item $(\tA\pf \tB)\pf \tC = \tA\pf (\tB\pf \tC)$;
    \item $(\tA\pf \tB)\cap (\tA\pf \tC)\subset \tA\pf (\tB\pf \tC)$.
\end{enumerate}
\end{cor}
\begin{proof}
For (1), it suffices to show $\Pi_{\tA\pf \tB} = \Pi_{\tA}\circ \Pi_{\tB}$. For each vertex $x\in \tX$, notice that both $\Pi_{\tA\pf \tB}(x)$  and $\Pi_{\tA}\circ \Pi_{\tB}(a)\in \Pi_{\tA}(\tB)=\tA\pf \tB$ are elements in $\tA\pf \tB$. So we just need to show that they are not separated by any hyperplane $H$ crossing $\tA\pf \tB$. 

Let $H$ be a hyperplane crossing $\tA\pf \tB$. Then it does not separate $x$ from $\tA\pf \tB$. Thus, by Bridge Theorem \ref{Thm: bridge theorem}, $x$ and its projection $\Pi_{\tA\pf \tB}(x)$ are in the same halfspace of $H$. 
Moreover, since $H$ crosses $\tA\pf \tB$, it crosses both $\tA$ and $\tB$. So it does not separate $x$ from $\Pi_{\tB}(x)$, and does not separate $\Pi_{\tB}(x)$ from $\Pi_{\tA}(\Pi_{\tB}(x))$.
As a result, the three points $x$, $\Pi_{\tA\pf \tB}(x)$, and $\Pi_{\tA}(\Pi_{\tB}(x))$ are in the same halfspace of $H$ for any hyperplane $H$ crossing $\tA\pf \tB$. This shows that $\Pi_{\tA\pf \tB}(x) = \Pi_{\tA}\circ \Pi_{\tB}(x)$ as desired for any vertex $x$.

Now, we prove (2). Note that the inclusion trivially holds if $(\tA\pf \tB)\cap (\tA\pf \tC)$ is empty. 
Suppose that there exists a vertex $a\in (\tA\pf \tB)\cap (\tA\pf \tC)$. Let $c = \Pi_{\tC}(a)\in \tC\pf \tA$ and $b = \Pi_{\tB}(c)\in \tB \pf \tC$. To show that $a\in \tA\pf (\tB\pf \tC)$, it suffices to prove that $a = \Pi_{\tA}(b)$, i.e. $a$ and $\Pi_{\tA}(b)$ are not separated by any hyperplane crossing $\tA$.

For each hyperplane $H$ crossing $\tA$, if $H$ does not cross $\tB$, then $H$ does not separate $\tA\pf \tB$ from $\tB$, and hence does not separate $a$ from $b$. If $H$ crosses $\tB$, then there are two cases:
\begin{itemize}
    \item If $H$ does not cross $\tC$, then $H$ does not separate $a\in \tA\pf\tC$ from $\tC$, and does not separate $b\in \tB\pf \tC$ from $\tC$. Thus, $a,b, \tC$ are in the same halfspace of $H$. 
    \item If $H$ crosses $\tC$, then $H$ does not separate $a$ from it projection $c = \Pi_{\tC}(a)$ to $\tC$. Similarly, since $H$ crosses $\tB$, it does not separate $b,c$. Hence, $a,b,c$ are in the same halfspace of $H$.
\end{itemize}
The above shows that $H$ does not separate $a,b$. Moreover, because $H$ crosses $\tA$, it does not separate $b$ from $\Pi_{\tA}(b)$. This concludes that $a,\Pi_{\tA}(b)$ are in the same halfspace of $H$ for each hyperplane $H$ crossing $\tA$.
\end{proof}

\begin{rmk}
Corollary \ref{corollary: property of projection}(1) also follows from \cite[Lemma~7.3.3]{Bowditch2024Median}, which deals with projections in the setting of median algebras. 

Note that Corollary \ref{corollary: property of projection}(1) means the operation $\pf$ is associative. Thus, we  denote $(\tA\pf \tB)\pf \tC = \tA\pf (\tB\pf \tC)$ simply by $\tA \pf \tB \pf \tC$ for brevity.
More generally, for a sequence $\tA_1,\dots, \tA_n$ of convex subcomplexes, we write 
\[
\pf_{i=1}^nA_i
= \tA_1\pf \tA_2 \pf \cdots \pf \tA_n
= \tA_1 \pf ( \tA_2 \pf ( \cdots  ( \pf \tA_{n-1}\pf \tA_n ) \cdots  ) ),
\]
which is a convex subcomplex of $A_1$.
\end{rmk}

\begin{rmk}\label{remark: gates correspond to intersection of subgroups}
Let $A\to X$ and $B\to X$ be local isometries of connected nonpositively curved cube complexes.
Denote by $\pi_1A^{g_1^{-1}}$ and $\pi_1B^{g_2^{-1}}$ the conjugate subgroups $g_1\cdot \pi_1A\cdot g_1^{-1}$ and $g_2\cdot \pi_1B\cdot g_2^{-1}$ for $g_1,g_2\in \pi_1X$.
Then the subgroup $\pi_1A^{g_1^{-1}}\cap \pi_1B^{g_2^{-1}}$ is represented by a local isometry $(g_1\tA\pf g_2\tB)/ \pi_1A^{g_1^{-1}}\cap \pi_1B^{g_2^{-1}}\to X$.

More precisely, let $\tA$ and $\tB$ be the based elevations of $A,B$ to the universal cover $\tX$. By Remark~\ref{remark: elevation and double cosets}, elevations of $A$ and $B$ to $\tX$ have the form $g_1\tA$ and $g_2\tB$ for $g_1,g_2\in \pi_1X$, with stabilizers $\stab(g_1\tA) = \pi_1A^{g_1^{-1}}$ and $\stab(g_2B) = \pi_1B^{g_2^{-1}}$. Thus, $\pi_1A^{g_1^{-1}} \cap \pi_1B^{g_2^{-1}}$ acts cocompactly on the projection $g_1\tA\pf g_2\tB$. The embedding $g_1\tA\pf g_2\tB\hookrightarrow \tX$ then induces a local isometry
\[
g_1\tA\pf g_2\tB/ \pi_1A^{g_1^{-1}} \cap \pi_1B^{g_2^{-1}}\to \tX/\pi_1X = X
\]
whose induced map of fundamental groups has image $\pi_1A^{g_1^{-1}} \cap \pi_1B^{g_2^{-1}}$.

More generally, given finitely many local isometries $A_1,\dots, A_n\to X$ of connected nonpositively curved cube complexes. Let $\tA_i$ be the based elevation of $A$ to $\tX$. Then for $g_1,\dots, g_n\in \pi_1X$, the subgroup $\pi_1A_1^{g_1^{-1}}\cap\cdots\cap \pi_1A_n^{g_n^{-1}}$ is represented by the local isometry 
\[
\pf_{i=1}^ng_iA_i
/ \pi_1A_1^{g_1^{-1}}\cap\cdots\cap \pi_1A_n^{g_n^{-1}} \to X.
\]
\end{rmk}

The following is proved in \cite[Proposition 3.11]{shepherd2025productseparabilityspecialcube}.
\begin{lem}\label{lemma: embedding of gates in special case}
Let $A,B$ be connected locally convex subcomplexes of a special cube complex $X$. Let $\tA, \tB$ be based elevations of $A,B$ to the universal cover $\tX$. For any $g_1,g_2\in \pi_1X$, we have
\[
\stab(g_1\tA)\cap \stab(g_2\tB) = \stab(g_1\tA\pf g_2\tB)
\]
and complex $(g_1\tA\pf g_2\tB)/\stab(g_1\tA\pf g_2\tB)$ embeds into $X$.
\end{lem}
\begin{rmk}
The statement of \cite[Proposition 3.11]{shepherd2025productseparabilityspecialcube} only assumes $X$ to be weakly special in the sense that no hyperplane self-intersects or self-osculates. See \cite[Definition 2.12]{shepherd2025productseparabilityspecialcube} for different notions of specialness.
\end{rmk}

\begin{cor}\label{corollary: gates are bounded}
Let $A_1,\cdots, A_n\to X$ be local isometries of connected compact virtually special cube complexes. Let $\tA_i$ be the based elevation of $A_i$ to $\tX$ for each $i$. Then the following hold:
\begin{enumerate}
    \item There exists a constant $S$ depending on local isometries $A_1,\cdots, A_n\to X$, such that for any integer $m>0$ and elements
    $g_1,\dots, g_m\in \pi_1X$, we have
    \[
    \diam\left( \pf_{i=1}^{m}g_i\tA_{n_i}/\cap_{i=1}^m\pi_1A_{n_i}^{g_i^{-1}} \right)
    \leq S,
    \]
    where the $A_{n_i}$ are (not necessarily) distinct elements of $\{A_1,\dots, A_n\}$.
    \item Among all finite intersections of $\pi_1X$-conjugates of $\pi_1A_1,\dots, \pi_1A_n$, there are finitely many $\pi_1X$-conjugacy classes.
\end{enumerate}
\end{cor}

\begin{proof}
We first pass to a finite-sheeted special cover $\hX$ of $X$ such that all elevations of $A_i$ to $\hX$ embed. To do this, we first take a compact special cover $\bX$ of $X$, and let $\bA_i$ be the based elevation of $A_i$ to $\bX$ for each $i$. Then take the canonical completion $C(\bA_i,\bX)$ of the local isometry $\bA_i\to \bX$. Finally, choose a compact regular cover $\hX$ of $X$ that factors through each $C(\bA_i,\bX)$, and let $\hA_i$ be based elevation of $\bA_i$ to $\hX$, which is also the based elevation of $A_i$.
Because $\bA_i$ embeds into $C(\bA_i,\bX)$, $\hA_i$ embeds into $\hX$. By regularity of $\hX$, each elevation $g\hA_i$ of $A_i$ to $\hX$ embeds into $\hX$ for $g\in \pi_1X$.

We now prove that $\left(\pf_{i=1}^{m}g_i\tA_{n_i}\right)/\left(\cap_{i=1}^{m}\pi_1\hA_{n_i}^{g_i^{-1}}\right)$ embeds into $\hX$ by induction on $m$. When $m = 1$, the statement holds because each $g\tA_i/\pi_1\hA_i^{g^{-1}}$ is an elevation $g\hA_i$ of $A_i$ to $\hX$, and hence embeds into $\hX$ by our construction. Suppose the the statement holds for $m-1$. Then for any $g_1,\dots, g_m\in \pi_1X$ and (not necessarily distinct) complexes $A_{n_1},\dots, A_{n_{m}}$, both $g_1\tA_{n_1}/\pi_1\hA_{n_1}^{g_1^{-1}}$ and $\left(\pf_{i=2}^{m}g_i\tA_{n_i}\right)/\left(\cap_{i=2}^{m}\pi_1\hA_{n_i}^{g_i^{-1}}\right)$ embed into $\hX$. Notice that their based elevations to $\tX$ are $g_1\tA_{n_1}$ and $\pf_{i=2}^{m}g_i\tA_{n_i}$ respectively. Thus, by Lemma \ref{lemma: embedding of gates in special case}, the complex 
\[
g_1\tA_{n_1}\pf \left(\pf_{i=2}^{m}g_i\tA_{n_i}\right)/\pi_1\hA_{n_m}^{g_{m}^{-1}}\cap \left(\cap_{i=1}^{m}\pi_1\hA_{n_i}^{g_i^{-1}}\right)
= \left(\pf_{i=1}^{m}g_i\tA_{n_i}\right)/\left(\cap_{i=1}^{m}\pi_1\hA_{n_i}^{g_i^{-1}}\right)
\]
embeds into $\hX$ as well. This finishes the induction.

To prove (1), let $S = \diam(\hX)$. The above implies that 
\[
\diam\left( \pf_{i=1}^{m}g_i\tA_{n_i}/\cap_{i=1}^m\pi_1\hA_{n_i}^{g_i^{-1}} \right)
\leq \diam(\hX) = S.
\]
Because $\pi_1\hA_i$ is a subgroup of $\pi_1A_i$ for each $i$, there is covering map 
\[
\pf_{i=1}^{m}g_i\tA_{n_i}/\cap_{i=1}^m\pi_1\hA_{n_i}^{g_i^{-1}} 
\twoheadrightarrow\ \pf_{i=1}^{m}g_i\tA_{n_i}/\cap_{i=1}^m\pi_1A_{n_i}^{g_i^{-1}}.
\]
In particular, for any integer $m>0$ and elements $g_1,\dots, g_m\in \pi_1X$, we have
\[
\diam\left( \pf_{i=1}^{m}g_i\tA_{n_i}/\cap_{i=1}^m\pi_1A_{n_i}^{g_i^{-1}} \right)
\leq\diam\left( \pf_{i=1}^{m}g_i\tA_{n_i}/\cap_{i=1}^m\pi_1\hA_{n_i}^{g_i^{-1}} \right)
\leq S.
\]

For (2), by Remark \ref{remark: gates correspond to intersection of subgroups}, 
each finite intersection $\cap_{i=1}^m\pi_1A_{n_i}^{g_i^{-1}}$ is represented by the local isometry $\left(\pf_{i=1}^{m}g_i\tA_{n_i}\right)/\left(\cap_{i=1}^m\pi_1A_{n_i}^{g_i^{-1}}\right)\to X$.
Since $\left(\pf_{i=1}^{m}g_i\tA_{n_i}\right)/\left(\cap_{i=1}^m\pi_1A_{n_i}^{g_i^{-1}}\right)$ has uniformly bounded diameter, there are finitely many isomorphism classes of such complexes. Because $X$ is compact, each isomorphism class admits finitely many possible local isometries to $X$. Thus, the finitely many isomorphism classes of complexes admit finitely many possible local isometries to $X$,
which induce finitely many $\pi_1X$-conjugacy classes of subgroups of the form $\cap_{i=1}^m\pi_1A_{n_i}^{g_i^{-1}}$. This finishes the proof of (2).
\end{proof}

\begin{rmk}
Corollary \ref{corollary: gates are bounded} may fail if we drop the assumption of virtual specialness. 
For example, Wise constructed a VH-complex $X$ in \cite[Part \RNum{2}, 2.1]{Wise1996NPCSquareComplexes} that is not virtually special.
The VH-complex $X$ has a structure of graph of graphs with a local isometry $X_e\to X$ from the edge space $X_e$ to $X$.
As shown in \cite[Part \RNum{2}, Corollary 3.4b]{Wise1996NPCSquareComplexes}, there exists an element $t\in \pi_1X$ such that the subgroups $\left\{ \cap_{i=0}^n \pi_1X_e^{t^i}\right\}_{n\geq 0}$ are free groups of distinct ranks, and hence represent infinitely many conjugacy classes of subgroups of $\pi_1X$.
\end{rmk}

We end this section with the following related result that will be used in the proof of Lemma \ref{lemma: vefity finite stature on paths}.
\begin{lem}\label{lemma: convex-cocopmact not conjugate to proper subgroup}
Let $A\to X$ be a local isometry of connected compact nonpositively curved cube complexes. Then the subgroup $\pi_1A$ of $\pi_1X$ is not conjugate to proper subgroups of $\pi_1A$.
\end{lem}

\begin{proof}
Let $\tA$ be the based elevation of $A$ to the universal cover $\tX$. Since $\pi_1A$ acts cocompactly on $\tA$, it acts cocompactly on a minimal $\pi_1A$-invariant convex subcomplex $\tC\subset \tA$ \cite[Lemma 3.7]{HuangJankiewiczPrzytycki2016Cocompactly}. Then for any $g\in \pi_1X$, $g\tC\subset g\tA$ is a minimal $\pi_1A^{g^{-1}}$-invariant convex subcomplex.

Suppose $\pi_1A^{g^{-1}}$ is a subgroup of $\pi_1A$ for some $g\in \pi_1X$. Then $\pi_1A^{g^{-1}} = \pi_1A\cap \pi_1A^{g^{-1}}$ acts cocompactly on the two projections $\tC\pf g\tC$ and $g\tC\pf \tC$. For the first projection $\tC\pf g\tC$, the inclusion $\tC\pf g\tC\hookrightarrow \tC$ induces a local isometry
\[
f: \tC\pf g\tC/ \pi_1A\cap \pi_1A^{g^{-1}} \to \tC/ \pi_1A,
\]
which factors through the covering map $p: \tC/\pi_1A\cap \pi_1A^{g^{-1}}\to \tC/ \pi_1A$. Thus, we obtain the following commutative diagram.
\[\begin{tikzcd}
	& {\tC/\pi_1A\cap \pi_1A^{g^{-1}}} \\
	{\tC\pf g\tC/ \pi_1A\cap \pi_1A^{g^{-1}}} & {\tC/ \pi_1A}
	\arrow["p", two heads, from=1-2, to=2-2]
	\arrow["i", hook, from=2-1, to=1-2]
	\arrow["f"', from=2-1, to=2-2]
\end{tikzcd}\]
For the second projection $g\tC\pf \tC$, since it is a convex subcomplex of $g\tC$, the minimality of $g\tC$ as a $\pi_1A^{g^{-1}}$-invariant convex subcomplex forces $g\tC = g\tC\pf \tC$. 
Hence, $g\tC\pf \tC/ \pi_1A\cap \pi_1A^{g^{-1}} = g\tC/\pi_1A^{g^{-1}}$.
The Bridge Theorem \ref{Thm: bridge theorem} and the isomorphism $\tC/\pi_1A \cong g\tC/\pi_1A^{g^{-1}}$ then give us
\[
\tC\pf g\tC/ \pi_1A\cap \pi_1A^{g^{-1}}
\cong g\tC\pf \tC/ \pi_1A\cap \pi_1A^{g^{-1}}
= g\tC/\pi_1A^{g^{-1}}
\cong \tC/\pi_1A.
\]
This implies that $f$ is a local isometry between isomorphic compact nonpositively curved cube complexes. Hence, $f$ must be an isomorphism \cite[Lemma 6.3]{Wise2004Sectional}. As a consequence, in the commutative diagram above, the composition $p\circ i \circ f^{-1} = 1$ is the identify map of $\tC/\pi_1A$. In particular, $p$ is $\pi_1$-surjective and $\pi_1A = \pi_1A\cap \pi_1A^{g^{-1}}$. This proves that the subgroup $\pi_1A^{g^{-1}}\leq \pi_1A$ cannot be proper.
\end{proof}

\section{Stature}\label{section: stature}
\subsection{Graph of cube complexes}\label{section: graph of cube complexes}
In this section, we review concepts related to graph of groups. See \cite{Serre1980} or \cite{ScottWall1979}.
\begin{defn}
Let $\Gamma$ be a graph with vertices $V(\Gamma)$ and edges $E(\Gamma)$.
A \textit{graph of nonpositively curved cube complexes} $X_\Gamma$ over $\Gamma$ consists of 
\begin{itemize}
    \item a vertex space $X_v$ that is a connected nonpositively curved cube complex for each vertex $v\in V(\Gamma)$;
    \item an edge space $X_e$ that is a connected nonpositively curved cube complex for each edge $e\in E(\Gamma)$;
    \item attaching maps $\varphi_{e}^-: X_e\to X_{\iota(e)}$ and $\varphi_e^+: X_e\to X_{\tau(e)}$, which are local isometries for each edge $e$ and its two endpoints $\iota(e), \tau(e)$.
\end{itemize}
The \textit{total space} of $X_\Gamma$ is defined to be the quotient space 
\[
\left( \bigsqcup_{v\in V(\Gamma)} X_v \right) \sqcup 
\left( \bigsqcup_{e\in E(\Gamma)} X_e\times I \right)/ \sim
\]
where $X_e\times \left\{0\right\} \sim \varphi_e^-(X_e)$ and $X_e\times \left\{1\right\} \sim \varphi_e^+(X_e)$. A cube complex $X$ is said to \textit{split as a graph of nonpositively curved cube complexes} if $X$ is the total space of a graph of nonpositively curved cube complexes.
\end{defn}

\begin{rmk}\label{remark: simplicial graph}
Notice that when the two endpoints $\iota(e)$ and $\tau(e)$ of an edge $e$ coincide, the attaching maps $\varphi_e^{-}$ and $\varphi_e^{+}$ correspond to the same edge space $X_e$ and the same vertex space $X_{\iota(e)} = X_{\tau(e)}$. This issue does not arise when the underlying graph $\Gamma$ is simplicial. For a general graph of nonpositively curved cube complexes, we may make the underlying graph simplicial by subdividing each loop edge. In the rest of the paper, we assume that the graph of nonpositively curved cube complexes $X_\Gamma$ is defined over a simplicial graph.
\end{rmk}

A graph of nonpositively curved cube complex $X_\Gamma$ has an associated \textit{graph of groups} $G_\Gamma$ consisting of a vertex group $\pi_1X_v$ and an edge group $\pi_1X_e$ for each vertex $v\in V(\Gamma)$ and edge $e\in E(\Gamma)$. The attaching maps $\varphi_e^-$ and $\varphi_e^+$ induce inclusions $(\varphi_e^-)_*:\pi_1X_e\hookrightarrow \pi_1X_{\iota(e)}$ and $(\varphi_e^+)_*: \pi_1X_e \hookrightarrow \pi_1X_{\tau(e)}$ for each edge $e\in E(\Gamma)$. 

Suppose $X$ splits as a graph of nonpositively curved cube complexes $X_\Gamma$. Then the universal cover $\tX$ is a union of copies of $\tX_v$ and $\tX_e\times I$ for each vertex $v\in V(\Gamma)$ and edge $e\in E(\Gamma)$, where $\tX_v, \tX_e$ are universal covers of the vertex space $X_v$ and the edge space $X_e$. The \textit{Bass--Serre tree} $\cT$ associated to $X_\Gamma$ is obtained from $\tX$ by identifying each copy of $\tX_v$ to a point and each copy of $\tX_e\times I$ to a copy of $I$. The action of $\pi_1X$ on $\tX$ then induces an action on $\cT$. For each vertex $\tv$ and edge $\te$, their stabilizers $\stab(\tv)$ and $\stab(\te)$ are conjugates of $\pi_1X_v$ and $\pi_1X_e$ in $\pi_1X$. In general, let $\pstab$ denote the pointwise stabilizer for the action of $\pi_1 X$ on $\cT$. For each finite path $\rho \subset \cT$, the group $\pstab(\rho)$ is represented by a local isometry in the following sense.
\begin{lem}\label{lemma: represent path stabilizer}
Suppose $X$ is a cube complex that splits as a graph of nonpositively curved cube complexes $X_\Gamma$ with the associated Bass--Serre tree $\cT$. Let $\tv$ be a lift of a vertex $v\in V(\Gamma)$ to $\cT$ such that $\stab(\tv) = \pi_1X_v$. Then for each finite path $\rho\subset \cT$ starting from $\tv$, the subgroup $\pstab(\rho)$ acts cocompactly on a convex subcomplex $\tY\subset \tX_v$
and is represented by a local isometry $\tY/\pstab(\rho)\to X_v$.
\end{lem}
\begin{proof}
The path $\rho:[0,n]\to \cT$ in the Bass--Serre tree starting from $\rho(0) = \tv$ corresponds to a subspace in the universal cover $\tX$ given by
\[
\left(\bigsqcup_{i=0}^n\tX_i\right) \sqcup \left( \bigsqcup_{j=1}^n \tY_i\times [0,1] \right)/\sim,
\]
where $X_0 = X_v, X_1,\dots, X_n$ and $Y_1,\dots, Y_n$ are (not necessarily distinct) vertex spaces and edge spaces.
For each $i$, $\tY_i\times \{0\}$ is attached to an elevation $\tY_i^0\subset \tX_{i-1}$ of $Y_i$ to $\tX_{i-1}$, and $\tY_i\times \{1\}$ is attached to an elevation $\tY_i^1\subset \tX_i$ of $Y_i$ to $\tX_i$. In particular, since each $\tY_i\times I$ separates other edge spaces, we have
\[
\tY_1^0 \pf \tY_n^0 
= \tY_1^0\pf \tY_2^0\pf \cdots \pf \tY_n^0.
\]
By considering the action of $\pi_1X$ on $\cT$ and $\tX$ respectively, we have
\[
\stab(\rho(i)) = \stab(\tX_i),\quad
\stab(\rho[i-1,i]) = \stab(\tY_i\times I) = \stab(\tY_i^0) = \stab(\tY_i^1).
\]
Then
\[
\pstab(\rho)
= \bigcap_{i=1}^n \stab(\rho[i-1,i])
= \bigcap_{i=1}^n \stab(\tY_i\times I)
= \bigcap_{i=1}^n \stab(\tY_i^0).
\]
Since each $\stab(\tY_i^0)$ acts cocopmactly on $\tY_i^0$, the pointwise stabilizer $\pstab(\rho)$ acts cocompactly on the projection 
\[
\tY_1^0\pf \tY_n^0
= \tY_1^0\pf \tY_2^0\pf \cdots \pf \tY_n^0.
\]
Then, by Remark \ref{remark: gates correspond to intersection of subgroups}, the subgroup $\pstab(\rho)\leq \stab(\tv)$ is represented by the local isometry 
\[
\tY_1^0\pf \tY_n^0/ \pstab(\rho) \to \tX_0/\stab(X_0) \cong X_0 = X_v.
\]
This proves the statement with $\tY = \tY_1^0\pf \tY_n^0$.
\end{proof}

\begin{cor}\label{corollary: represent tree stabilizer}
Suppose $X$ is a cube complex that splits as a graph of nonpositively curved cube complexes $X_\Gamma$ with the associated Bass--Serre tree $\cT$. Let $\tv$ be a lift of a vertex $v\in V(\Gamma)$ to $\cT$ such that $\stab(\tv) = \pi_1X_v$. Then for each finite subtree $T\subset \cT$ containing $\tv$, the subgroup $\pstab(T)$ is represented by a local isometry $\tY/\pstab(T)\to X_v$ for some convex subcomplex $\tY\subset \tX_v$.
\end{cor}
\begin{proof}
The finite subtree $T$ is a union of finitely many paths $\rho_1,\dots, \rho_n$ starting from $\tv$. So the pointwise stabilizer is a finite intersection $\pstab(T) = \cap_{i=1}^n \pstab(\rho_i)$.
By Lemma~\ref{lemma: represent path stabilizer}, each pointwise stabilizer $\pstab(\rho_i)$ is represented by a local isometry $\tY_i/\pstab(\rho_i)\to X_v$, where $\pstab(\rho_i)$ acts cocompactly on the convex subcomplex $\tY_i\subset \tX_v$. Thus, the intersection $\pstab(T) = \cap_{i=1}^n \pstab(\rho_i)$ acts cocompactly on the projection 
\[
\tY = \tY_1\pf \tY_2\pf \dots \pf \tY_n.
\]
The isometric embedding $\tY\hookrightarrow \tX_v$ then induces a a local isometry
\[
\tY/\pstab(T) \to \tX_v/\pi_1X_v = X_v,
\]
which represents the subgroup $\pstab(T)\leq \pi_1X_v$.
\end{proof}

\subsection{Finite stature}
We review the following notion from \cite{Huang--Wise2019statureseparabilitygraphsgroups}.
\begin{defn}\label{definition: finite stature}
Let $G$ be a group and let $\Omega = \{G_\lambda\}_{\lambda\in \Lambda}$ be a collection of subgroups of $G$. We say that $G$ has \textit{finite stature} with respect to $\Omega$ if for each $H\in \Omega$, there are finitely many $H$-conjugacy classes of infinite subgroups of the form $H\cap C$, where $C$ is an intersection of $G$-conjugates of elements of $\Omega$.
\end{defn}
Suppose that a cube complex $X$ splits as a graph of nonpositively curved cube complexes $X_\Gamma$. The fundamental group $\pi_1X$ then splits as a graph of groups. We are particularly interested in the finite stature of $\pi_1X$ with respect to the vertex groups $\{\pi_1X_v\mid v\in V(\Gamma)\}$.
The following theorem, due to Huang and Wise \cite[Theorem~1.4]{Huang--Wise2024fintiestature}, shows that in this setting, finite stature is related to virtual specialness.
\begin{thm}\label{theorem: fintie stature equivalent to virtual specialness}
    Let $X$ be a compact cube complex that splits as a graph of nonpositively curved cube complexes. Suppose each vertex group is hyperbolic. Then the following are equivalent.
    \begin{enumerate}
        \item $\pi_1X$ has finite stature with respect to the vertex groups.
        \item $X$ is virtually special.
    \end{enumerate}
\end{thm}
Recall that for each vertex group $\pi_1X_v$, its conjugate $\pi_1X_v^g$ in $\pi_1X$ corresponds to the stabilizer of a vertex $\tv$ in the associated Bass--Serre tree $\cT$. 
So the intersection $H \cap C$ in Definition~\ref{definition: finite stature} corresponds to the pointwise stabilizer $\pstab(T)$ of some (possibly infinite) subtree $T \subset \cT$.
Therefore, to establish finite stature, it suffices to consider pointwise stabilizers of subtrees (see \cite[Lemma 3.9]{Huang--Wise2019statureseparabilitygraphsgroups}). The following result shows that it is enough to restrict attention to pointwise stabilizers of finite paths.
\begin{lem}\label{lemma: vefity finite stature on paths}
Let $X$ be a compact cube complex that splits as a graph of virtually special cube complexes. Let $\cT$ be the associated Bass--Serre tree. Then the following are equivalent.
\begin{enumerate}
    \item $\pi_1X$ has finite stature with respect to the vertex groups.
    \item For each vertex $v\in V(\Gamma)$, there are finitely many $\pi_1X_v$-conjugacy classes of infinite subgroups of the form $\pi_1X_v\cap (\cap_{e\in \cE}\stab(e))$, where $\cE\subset E(\cT)$.
    \item For each vertex $v\in V(\Gamma)$ and a lift $\tv\in V(\cT)$ with $\stab(\tv) = \pi_1X_v$, there are finitely many $\pi_1X_v$-conjugacy classes of infinite subgroups of the form $\pstab(T)$, where $T\subset \cT$ is a subtree containing $\tv$.
    \item For each vertex $v\in V(\Gamma)$ and a lift $\tv\in V(\cT)$ with $\stab(\tv) = \pi_1X_v$, there are finitely many $\pi_1X_v$-conjugacy classes of infinite subgroups of the form $\pstab(T)$, where $T\subset \cT$ is a finite subtree containing $\tv$.
    \item For each vertex $v\in V(\Gamma)$ and a lift $\tv\in V(\cT)$ with $\stab(\tv) = \pi_1X_v$, there are finitely many $\pi_1X_v$-conjugacy classes of infinite subgroups of the form $\pstab(\rho)$, where $\rho$ is a finite path in $\cT$ starting from $\tv$.
\end{enumerate}
\end{lem}
\begin{proof}
The equivalence of (1) and (2) is proved in \cite[Lemma 3.9]{Huang--Wise2019statureseparabilitygraphsgroups}.
Note that (2) obviously implies (5). We show (5) $\Longrightarrow$ (4) $\Longrightarrow$ (3) $\Longrightarrow$ (2).

Assume (3). For each collection of edges $\cE\subset E(\cT)$, take the convex hull $T = \Hull(\cE\cup \{\tv\})$, which is a subtree of $\cT$. Then
\[
\pi_1X_v\cap (\cap_{e\in \cE}\stab(e)) 
= \stab(\tv) \cap (\cap_{e\in \cE}\stab(e)) 
= \pstab(T).
\]
The finitely many $\pi_1X_v$-conjugacy classes of infinite subgroups $\pstab(T)$ therefore yield finitely many $\pi_1X_v$-conjugacy classes of infinite subgroups in (2).

Suppose (4) holds. For each subtree $T\subset \cT$, we can write it as an increasing union $T = \cup_{i=1}^\infty T_i$, where $T_i\subseteq T_{i+1}$ and each $T_i$ is a finite subtree containing $\tv$. Then $\pstab(T)
= \cap_{i=1}^\infty \pstab(T_i)$ is a decreasing intersection of subgroups
\begin{equation}\label{equation: descending chain}
\pstab(T_1)\supset \pstab(T_2)\supset \cdots .
\end{equation}
We claim that the above descending chain (\ref{equation: descending chain}) has only finitely many proper inclusions. Otherwise, suppose that $\pstab(T_{k_i}) \supsetneq \pstab(T_{k_i+1})$ for infinitely many integers $\{k_i\}_{i=1}^\infty$ with $k_i < k_{i+1}$.
By Corollary \ref{corollary: represent tree stabilizer}, each $\pstab(T_i)$ is represented by a local isometry of virtually special cube complexes $\tY_i/\pstab(T_i)\to X_v$. 
It then follows from Lemma \ref{lemma: convex-cocopmact not conjugate to proper subgroup} that $\pstab(T_{k_i})$ is not $\pi_1X_v$-conjugate to $\pstab(T_{n})$ for any $n>k_i$. Thus, the collection of subgroups $\{\pstab(T_{k_i})\}_{i=1}^\infty$ represents infinitely many distinct $\pi_1X_v$-conjugacy classes of subgroups.
This contradicts the assumption of (4). Therefore, the descending chain (\ref{equation: descending chain}) has only finitely many proper inclusions. In particular, $\pstab(T) = \pstab(T_n)$ for some large enough $n$. As a result, the $\pi_1X_v$-conjugacy classes of infinite subgroups in (3) coincide with the $\pi_1X_v$-conjugacy classes in (4), which are finite by assumption. This shows that (4) $\Longrightarrow$ (3).

Finally, we prove (5) $\Longrightarrow$ (4). Notice that each finite subtree $T$ containing $\tv$ is the union of finitely many paths $T = \rho_1\cup \cdots \cup \rho_m$, where each $\rho_i$ starts from $\tv$. So $\pstab(T)$ is a finite intersection $\pstab(T) = \bigcap_{i=1}^m \pstab(\rho_i)$. Choose finite paths $\brho_1,\cdots, \brho_n$ whose pointwise stabilizers represent the $\pi_1X_v$-conjugacy classes in (5). 
By Lemma \ref{lemma: represent path stabilizer}, each $\pstab(\brho_i)$ is represented by a local isometry $A_i = \tY_i/\pstab(\brho_i)\to X_v$ for some convex subcomplex $\tY_i\subset \tX_v$.
Since each infinite $\pstab(\rho_i)$ belongs to the $\pi_1X_v$-conjugacy class of $\pstab(\brho_{n_i})$ for some $1\leq n_i\leq n$, there exist elements $g_1,\dots, g_m\in \pi_1X_v$ such that
\[
\pstab(T)
= \bigcap_{i=1}^m\pstab(\rho_i)
= \bigcap_{i=1}^m \pstab(\brho_{n_i})^{g_i^{-1}}
= \bigcap_{i=1}^m \pi_1A_{n_i}^{g_i^{-1}}.
\]
By Corollary \ref{corollary: gates are bounded}(b), the above finite intersection has only finitely many $\pi_1X_v$-conjugacy classes of infinite subgroups. This proves (4).
\end{proof}
\begin{rmk}
When the vertex groups are hyperbolic and the edge groups are quasiconvex subgroups, Lemma~\ref{lemma: vefity finite stature on paths} is proved in \cite[Lemma~3.19]{Huang--Wise2019statureseparabilitygraphsgroups} and \cite[Proposition~3.3]{jankiewicz2023finitestatureartingroups}. The proofs use the fact that quasiconvex subgroups of a hyperbolic group have finite height \cite{GitikMitraRipsSageev1998}.
\end{rmk}

\section{Quasilines in $\CAT(0)$ cube complexes}\label{section: quasiline}
\subsection{Hyperplanes in quasilines}
\begin{defn}
Let $X$ be a compact nonpositively curved cube complex and let $\tX$ be its universal cover. 
A \textit{quasiline} in $\tX$ is a pair $(\tY,\phi)$ consisting of a convex subcomplex $\tY\subset \tX$ and a nontrivial element $\phi\in\pi_1X$ such that the subgroup $\la\phi\ra$ acts cocompactly on $\tY$. Two quasilines $(\tY_1,\phi_1), (\tY_2,\phi_2)$ are \textit{isomorphic} in $\tX$ if there exists $g\in \pi_1X$ such that $(\tY_2,\phi_2) = (g\tY_1,g\phi_1g^{-1})$.
\end{defn}
\begin{rmk}\label{remark: hyperbolic elements of quasiline}
Since $\tY$ admits a geometric action by the infinite cyclic group $\la\phi\ra\cong \bZ$, it is quasi-isometric to $\bZ$ and hence has two ends.
By Remark \ref{remark: elements are hyperbolic}, $\phi$ is a hyperbolic isometry and has positive combinatorial translation length $|\phi|>0$. In particular, $\phi$ preserves two ends of $\tY$.
\end{rmk}
\begin{exam}
Let $X$ be a compact nonpositively curved cube complex and let $G_1,\cdots, G_n$ be subgroups of $\pi_1X$. Suppose each $G_i$ acts cocompactly on a convex subcomplex $\tY_i$ of $\tX$, then the intersection $G_1\cap \cdots\cap G_n$ acts cocompactly on the projection $\tY_1\pf \cdots \pf \tY_n$. If $G_1\cap \cdots\cap G_n \cong \bZ$, then by choosing a generator $\phi$ of $\cap_{i=1}^nG_i$, we obtain a quasiline $\left(\pf_{i=1}^n\tY_i, \phi\right)$ in $\tX$.
\end{exam}

We recall the following from \cite{CapraceSageev2011RankRigidity}.
\begin{defn}
Let $H$ be a hyperplane in a $\CAT(0)$ cube complex $\tX$. A halfspace of $H$ is said to be \textit{deep} if it contains points arbitrarily far away from $H$, and \textit{shallow} otherwise. The hyperplane $H$ is \textit{essential} if both of its halfspaces $H^\pm$ are deep, \textit{trivial} if both are shallow, and \textit{half-essential} if exactly one of them is deep. 
\end{defn}

\begin{lem}\label{lemma: characterization of three types of hyperplanes}
Let $X$ be a compact nonpositively curved cube complex. Suppose $(\tY,\phi)$ is a quasiline in $\tX$. Then for each hyperplane $H$ in $\tY$, we have:
\begin{enumerate}
    \item $H$ is trivial if and only if $\diam(H) = \infty$.
    \item $H$ is half-essential if and only if it has a compact halfspace.
    \item $H$ is essential if and only if $\diam(H)<\infty$ and $H$ has only noncompact halfspaces.
\end{enumerate}
\end{lem}
\begin{proof}
For (1), note that a hyperplane $H$ is trivial if and only if $\tY$ lies in the $r$-neighborhood $N_r(H)$ of $H$ for some constant $r$. Because the quasiline $\tY$ has diameter $\diam(\tY) = \infty$, the inclusion $\tY\subset N_r(H)$ holds only if $\diam(H) = \infty$. 

Conversely, assume $\diam(H) = \infty$. Because $\tY$ is $\la \phi \ra$-cocompact, there exists a compact subcomplex $K\subset \tY$ such that $\la\phi\ra\cdot K = \tY$. Since $H$ is infinite, $H\cap \phi^n(K)\neq\emptyset$ for infinitely many $n$. Thus, $K$ intersects the hyperplanes $\phi^{-n}(H)$ for infinitely many integers $n$. However, the compact complex $K$ intersects only finitely many hyperplanes in $\tY$. So there exist integers $m\neq n$ such that $\phi^{-m}(H)$ and $\phi^{-n}(H)$ are the same hyperplane and intersect $K$. In particular, $\phi^{k}\in \stab(\phi^{-n}(H))$, where $k = n-m$.
Because $\la\phi\ra\cdot K = \tY$, by letting $K' = \bigcup_{i=0}^{|k|-1}\phi^i(K)$, we have $\la\phi^k\ra\cdot  K' = \tY$. Now take $r = \diam(K')$. Choose a point $x\in K\cap \phi^{-n}(H)\subset K'\cap \phi^{-n}(H)$, then
\[
\tY
= \la\phi^k\ra\cdot  K' 
\subset \la\phi^k\ra\cdot N_{r}\left(  x\right) 
= N_{r}\left(\la\phi^k\ra\cdot  x\right) 
\subset N_r(\phi^{-n}(H)).
\]
Thus, 
\[
\tY = \phi^n(\tY)\subset \phi^n(N_r(\phi^{-n}(H))) = N_r(H),
\]
which implies that $H$ is trivial. This finishes the proof of (1).

For (2), let $H$ be a half-essential hyperplane. Then the shallow halfspace of it, say $H^+$, is contained in the $k$-neighborhood $N_k(H)$ of $H$ for some constant $k$. By (1), $\diam(H)<\infty$ and so $H^+$ is compact. Conversely, suppose $H^+$ is a compact halfspace of $H$. Because $H$ is contained in the $1$-neighborhood $N_1(H^+)$, we have $\diam(H)<\infty$, and $H$ cannot be essential as $H^+$ is not deep.

Finally, let $H$ be an essential hyperplane. Its halfspaces are both deep and therefore noncompact, and by (1) we have $\diam(H) < \infty$.
Conversely, finite diameter implies that $H$ is not trivial. Since it only has noncompact halfspaces, (2) implies that $H$ is not half-essential. This proves (3).
\end{proof}

\begin{lem}\label{lemma: properties of three types of hyperplanes}
Let $X$ be a compact nonpositively curved cube complex.
Suppose $(\tY,\phi)$ is a quasiline in $\tX$, then 
\begin{enumerate}
    \item there exists a constant $D$ such that $\diam(H)\leq D$ for each hyperplane $H$ in $\tY$ that is not trivial;
    \item there exists a constant $K$ such that for the shallow halfspace $H^\epsilon$ of each half-essential hyperplane $H$, $\diam(H^\epsilon)\leq K$;
    \item there exists an integer $d$ such that for each essential hyperplane $H$, $H$ and $\phi^d(H)$ do not intersect, and are only simultaneously crossed by trivial hyperplanes;
    \item $\tY$ has only finitely many trivial hyperplanes.
\end{enumerate}
The constants $K,D,d$ depend only on the quasiline $(\tY,\phi)$.
Moreover, any constants $\bD\geq D, \bK\geq K$ and integer $\bd\geq d$ also satisfy (1), (2), (3) respectively.
\end{lem}
\begin{proof}
The cocompact action of $\la \phi \ra$ on $\tY$ induces an action on hyperplanes of $\tY$ with finitely many orbits, say $\la \phi\ra \cdot H_1,\cdots, \la \phi\ra \cdot H_n$. By Lemma \ref{lemma: characterization of three types of hyperplanes}, only trivial hyperplanes have infinite diameter, and only shallow halfspaces of half-essential hyperplanes have finite diameters.
By taking
\[
D = \max_{1\leq i\leq n}\{\diam(H_i)\mid \diam(H_i)<\infty\}
\]
and 
\[
K = \max_{1\leq j\leq n}\{\diam(H_j^\epsilon)\mid \diam(H_j^\epsilon)<\infty,\ \text{$H_j^\epsilon$ is a halfspace of $H_j$}\},
\]
we obtained the desired $D$ in (1) and $K$ in (2). 

For each essential hyperplane $H$, the distance between its carrier $N(H)$ and the translation $\phi^d(N(H))$ satisfies
\[
d(N(H),\phi^{d}(N(H)))
\geq d \lv\phi\rv - 2 \diam(N(H)) 
\geq d\lv \phi \rv - 2D - 2.
\]
Choose $d$ large enough such that $d\lv \phi \rv > 3D+2$. Then,
$d\lv \phi \rv - 2D- 2> D>0$. Thus, $H$ and $\phi^d(H)$ do not intersect, and are only simultaneously crossed by hyperplanes of diameter $> D$. By our choice of $D$ in (1), these hyperplanes are trivial. This proves (3).

For (4), note that every essential hyperplane crosses every trivial hyperplane. In fact, if an essential hyperplane $H$ and a trivial hyperplane $J$ do not cross, a halfspace $J^{\epsilon_J}$ of $J$ is contained in a halfspace $H^{\epsilon_H}$ of $H$. But then $H$ has a deep halfspace, since $J^{\epsilon_J}$ contains points arbitrarily far from $J$, and hence arbitrarily far from $H$.
    
Because each essential hyperplane $H$ has finite diameter, it intersects finitely many hyperplanes. In particular, $\tY$ has finitely many trivial hyperplanes.

Finally, notice that our construction of $D,K,d$ depend only on the quasiline $(\tY, \phi)$.
Moreover, note that any constants $\bD\geq D$, $\bK\geq K$, $\bd\geq d$ also apply to the above argument, and therefore satisfy (1), (2), (3) respectively. This finishes the proof.
\end{proof}

\begin{lem}\label{lemma: essential hyperplane crosses long geodesic}
Let $X$ be a compact nonpositively curved cube complex.
Suppose that $(\tY,\phi)$ is a quasiline in $\tX$. Then there exist constants $N,M$ depending only on $(\tY,\phi)$, such that for any integer $k>0$, we have the following:
\begin{enumerate}
    \item $\diam(H^+\cap \phi^{\pm k}(H^-))\leq k N$ for any essential hyperplane $H$ in $\tY$.
    \item Suppose that an essential hyperplane $H$ in $\tY$ crosses a combinatorial geodesic $\gamma$ at an edge $e_0$. If $e_0$ has distance $\geq kM$ from two endpoints of $\gamma$, then $\phi^{\pm k}(H)$ crosses $\gamma$.
\end{enumerate}
Moreover, any constants $\bM\geq M, \bN\geq N$ also satisfy (1) and (2) respectively.
\end{lem}
\begin{proof}
For each essential hyperplane $H$ of $\tY$, because its two halfspaces in $\tY$ are deep, it separates two ends of $\tY$. By Remark \ref{remark: hyperbolic elements of quasiline}, $\phi$ preserves two ends of $\tY$. In particular, the subspaces $H^+\cap \phi(H^+)$ and $H^-\cap \phi(H^-)$ are noncompact and contain two ends of $\tY$ respectively. We claim that the subspaces $H^+\cap \phi(H^-)$ and $H^-\cap \phi(H^+)$ are both compact. Indeed, since $\diam(H)<\infty$, if either $H^+\cap \phi(H^-)$ or $H^-\cap \phi(H^+)$ is noncompact, it contains an additional end of $\tY$, which contradicts the fact that $\tY$ has two ends. See Figure~\ref{figure: three ends of quasiline} for the case where $H^+ \cap \phi(H^-)$ is noncompact.
\begin{figure}[h]
    \centering
    \includegraphics[width=0.75\linewidth]{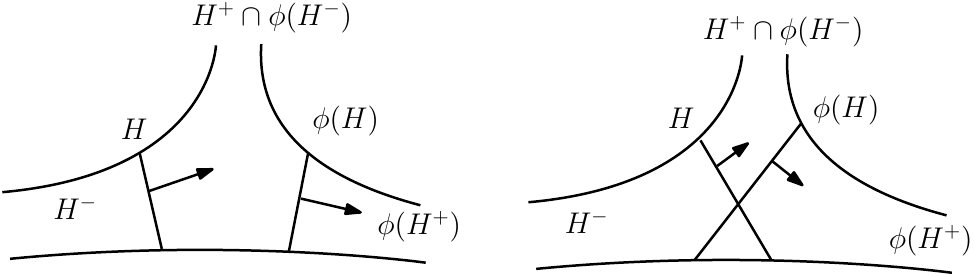}
    \caption{If the subspace $H^+ \cap \phi(H^-)$ is noncompact, then it contains an end of the quasiline.}
    \label{figure: three ends of quasiline}
\end{figure}

The cocompact action of $\la\phi\ra$ on $\tY$ induces an action on the set of essential hyperplanes of $\tY$ with finitely many orbits $\la\phi\ra\cdot H_1,\cdots, \la\phi\ra\cdot H_n$. Choose $D$ as in Lemma \ref{lemma: properties of three types of hyperplanes}(1). In particular, $\diam(N(H_i))\leq D+1$ for each $i$.
Let
\[
N = \max_{1\leq i\leq n} \left\{ \diam(H_i^-\cap \phi(H_i^+)),\ \diam(H_i^+\cap \phi(H_i^-)) \right\} + D + 1
\]
which is finite because $H_i^-\cap \phi(H_i^+)$ and $H_i^+\cap \phi(H_i^-)$ are both compact for each $1\leq i\leq n$.

We prove (1) by induction on $k$. When $k = 1$, our construction of $N$ gives
\[
\diam \left( H^+\cap \phi(H^-) \right)<N
\]
and 
\begin{align*}
\diam \left( H^+\cap \phi^{-1}(H^-) \right)
&= \diam \left( \phi\left(H^+\cap \phi^{-1}(H^-) \right)\right)\\
&= \diam \left( \phi\left(H^+\right)\cap H^- \right)<N.
\end{align*}
Suppose the statement holds for $k-1$. Because
\begin{align*}
H^+\cap \phi^k(H^-)
&= H^+\cap \phi^k(H^-) \cap \left( \phi^{k-1}(H^-)\cup \phi^{k-1}(N(H))\cup \phi^{k-1}(H^+) \right)\\
&\subset \left(H^+\cap \phi^{k-1}(H^-)\right) \cup \phi^{k-1}(N(H)) \cup \left( \phi^{k-1}(H^+) \cap \phi^{k}(H^-) \right)\\
&= \left(H^+\cap \phi^{k-1}(H^-)\right) \cup \phi^{k-1}(N(H)) \cup \phi^{k-1}\left( H^+ \cap \phi(H^-) \right),
\end{align*}
we have
\begin{align*}
\diam\left( H^+\cap \phi^k(H^-) \right)
&\leq \diam\left(H^+\cap \phi^{k-1}(H^-)\right) + \diam\left( \phi^{k-1}(N(H))\right)\\ 
&\quad + \diam\left( \phi^{k-1}\left( H^+ \cup \phi(H^-) \right) \right)\\
&\leq (k-1)N + (D+1) + \diam\left( H^+ \cup \phi(H^-) \right)\\
&\leq kN.
\end{align*}
Similarly, by exchanging $H^-$ and $H^+$ in the above, we have $\diam\left( H^-\cap \phi^k(H^+) \right)\leq kN$. Thus, 
\begin{align*}
\diam\left( H^+\cap \phi^{-k}(H^-) \right)
&= \diam\left(\phi^k\left( H^+\cap \phi^{-k}(H^-) \right)\right)\\
&= \diam\left(\phi^k\left( H^+\right)\cap H^- \right)\leq kN.
\end{align*}
This finishes the proof of (1). 

For (2), choose $d$ as in Lemma \ref{lemma: properties of three types of hyperplanes}(3). In particular, $H$ does not intersect $\phi^{d+1}(H)$ and $\phi^{-d}(H)$. Recall that since $\phi$ preserves two ends of $\tY$, by exchanging $H^+$ and $H^-$ if necessary, we may assume that $\phi^{d+1}(H^+)\subset H^+ \subset \phi^{-d}(H^+)$ as shown in Figure \ref{figure: essential hyperplanes separate long geodesics}. Let $M = (d+2)N$. We prove (2) by induction on $k$. When $k = 1$, by (1), we have
\[
\diam(H^+\cap \phi^{d+1}(H^-))
\leq (d+1)N,
\quad
\diam(H^-\cap \phi^{-d}(H^+))\leq dN.
\]
Because $e_0$ has distance $\geq M > (d+1)N$, two endpoints of $\gamma$ lie in $\phi^{d+1}(H^+)$ and $\phi^{-d}(H^-)$ respectively. Now consider the hyperplane $\phi(H)$. By Lemma \ref{lemma: properties of three types of hyperplanes}(3) and our choice of $d$, $\phi(H)$ does not intersect $\phi^{d+1}(H)$ and $\phi^{-d}(H)$. Hence, we have nested halfspaces
\[
\phi^{d+1}(H^+)\subset \phi(H^+) \subset \phi^{-d}(H^+).
\]
In particular, two ends of $\gamma$ are contained in $\phi^{d+1}(H^+)\subset \phi(H^+)$ and $\phi^{-d}(H^-)\subset \phi(H^-)$, and $\phi(H)$ crosses $\gamma$. A similar argument shows that $\phi^{-1}(H)$ crosses $\gamma$.
This proves (2) when $k=1$.

Suppose the statement holds for $k-1$. If $e_0$ has distance $\geq kM$ from two endpoints. Then $\phi^{k-1}(H)$ crosses $\gamma$ at an edge $e_{k-1}$. Furthermore, by (1), we have 
\[
d(e_0,e_k)
\leq \diam\left( H^+\cap \phi^{k-1}(H^-) \right)\leq (k-1)N < (k-1)M.
\]
Hence, $e_k$ has distance $> kM - (k-1)M = M$ from two endpoints of $\gamma$. By our argument in the case $k=1$, the hyperplane $\phi^k(H) = \phi(\phi^{k-1}(H))$ crosses $\gamma$. A similar argument shows that $\phi^{-k}(H)$ crosses $\gamma$. This finishes the induction and completes the proof of (2).

Finally, since the constants $D,d$ in Lemma \ref{lemma: characterization of three types of hyperplanes} depend only on the quasiline $(\tY,\phi)$, it follows from our construction of $N$ and $M$ that they also depend only on $(\tY,\phi)$.
Moreover, any constants $\bN\geq N$ and $\bM\geq M$ also apply to the above argument, and hence satisfy (1) and (2) respectively. 
\end{proof}
\begin{figure}[h]
    \centering
    \includegraphics[width=0.75\linewidth]{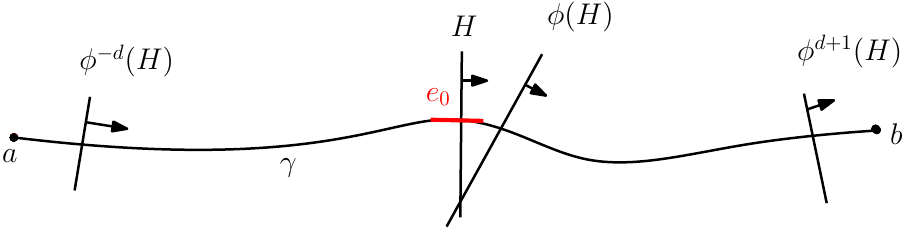}
    \caption{Four hyperplanes cross the combinatorial geodesic $\gamma$. The arrows denote halfspaces $\phi^{-d}(H^+), H^+, \phi(H^+), \phi^{d+1}(H^+)$. $\gamma$ crosses $H$ at an edge $e_0$, whose distance from each endpoint $a,b$ of $\gamma$ is $\geq M$.}
    \label{figure: essential hyperplanes separate long geodesics}
\end{figure}

\subsection{Convex subcomplexes in quasilines fellow-traveling with a geodesic}
\begin{lem}\label{lemma: property of single convex subcopmlex}
Let $X$ be a compact nonpositively curved cube complex and let $(\tY,\phi)$ be a quasiline in $\tX$. Suppose $H$ is an essential hyperplane in $\tY$ such that $H$ and $\phi(H)$ do not intersect and are not simultaneously crossed by any hyperplane of $\tY$.
Choose a vertex $a\in N(H)\pf \phi(N(H))$. Consider the convex subcomplex $C = \Hull\{a, \phi^2(a)\}$.
Then the following hold.
\begin{enumerate}
    \item Let $\gamma$ be a combinatorial geodesic in $\tY$ that crosses the hyperplanes $\phi^i(H)$ at edges $e_i$ for $i = 0,1,2,3$. Denote by $\gamma_{[e_i,e_j]}$ the subsegment of $\gamma$ between $e_i,e_j$. Then 
    \[
    \gamma_{[e_1,e_2]}\subset C\subset \Hull\{e_0,e_3\}.
    \]
    \item For any integer $n\geq 0$, we have $\cup_{i=0}^n\phi^i(C) = \Hull\{a,\phi^{n+2}(a)\}$. In particular, the subcomplex $\cup_{i=0}^n\phi^i(C)$ is convex.
\end{enumerate}
\end{lem}
\begin{proof}
Because $H$ and $\phi(H)$ do not intersect and are not simultaneously crossed by any hyperplane, the Bridge Theorem~\ref{Thm: bridge theorem} implies that the projection $N(H) \pf \phi(N(H))$ is a singleton. 
So we must have $\{a\} = N(H)\pf \phi(N(H))$. Denote the other projection by $\{b\} = N(H)\pf \phi(N(H))$. 
By Remark \ref{remark: hyperbolic elements of quasiline}, $\phi$ reserves two ends of $\tY$. So either $\phi(H^+)\subset H^-$ or $\phi(H^+)\subset H^-$. Exchange $H^+$ and $H^-$ if necessary, we may assume that $\phi(H^+)\subset H^+$.
Then we have nested halfspaces $\phi^3(H^+)\subset \phi^2(H^+) \subset \phi(H^+)\subset H^+$.
Let endpoints of $e_i$ in $\phi^i(H)^-$ and $\phi^i(H)^+$ be $u_i$ and $v_i$ respectively, and the endpoints $u_i,v_i$ alternate along $\gamma$. See Figure \ref{figure: subsegment contained in C}(A).
\begin{figure}[h]
\begin{subfigure}{0.45\textwidth}
    \includegraphics[width=\linewidth]{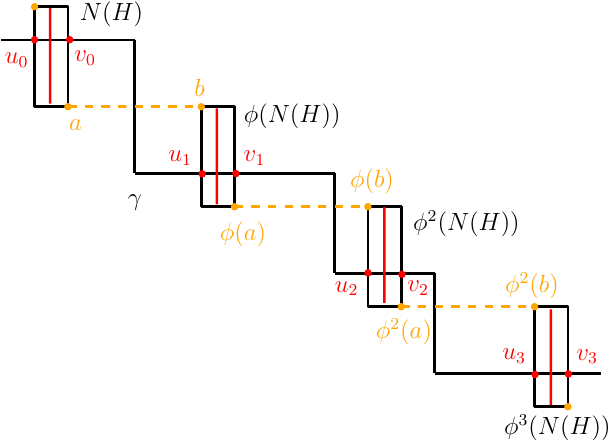}
    \caption{}
    \label{figure: subsegment contained in C(a)}
\end{subfigure}
\begin{subfigure}{0.45\textwidth}
    \includegraphics[width=\linewidth]{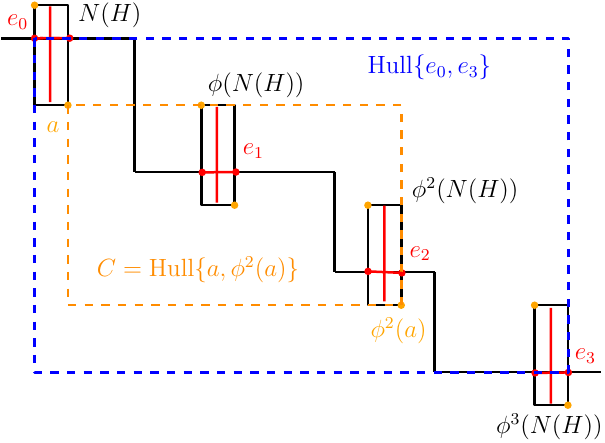}
    \caption{}
    \label{figure: subsegment contained in C(b)}
\end{subfigure}
    \caption{The combinatorial geodesic $\gamma$ crosses four hyperplanes. Points $u_i, v_i$ alternate along $\gamma$ as shown in (A). The two convex subcomplexes $C = \Hull\{a,\phi^2(a)\}$ and $\Hull\{e_0,e_3\}$ are illustrated as the orange and blue regions in (B).}
    \label{figure: subsegment contained in C}
\end{figure}

To prove (1), it suffices to show that $e_1,e_2\subset C$. We first consider the subsegment $\gamma_{[v_{i}, u_{i+1}]}$ from $v_{i}$ to $u_{i+1}$ for $i = 0,1,2$. Choose following combinatorial geodesics (the choices may not be unique):
\begin{itemize}
    \item a combinatorial geodesic $[v_{i},\phi^{i}(a)]$ from $v_{i}$ to $\phi^{i}(a)$,
    \item a combinatorial geodesic $[\phi^{i}(a),\phi^i(b)]$ from $\phi^i(a)$ to $\phi^i(b)$,
    \item a combinatorial geodesic $[\phi^i(b), u_{i+1}]$ from $\phi^i(b)$ to $u_{i+1}$.
\end{itemize}
Note that, $[v_i, \phi^i(a)]$ crosses hyperplanes intersecting $\phi^i(H)$, $[\phi^{i}(a),\phi^i(b)]$ crosses hyperplanes separating $\phi^i(H)$ and $\phi^{i+1}(H)$ (by Bridge Theorem \ref{Thm: bridge theorem}), and $[\phi^i(b), u_{i+1}]$ crosses hyperplanes intersecting $\phi^{i+1}(H)$.
Since $\phi^i(H)$ and $\phi^{i+1}(H)$ are not simultaneously crossed by any hyperplane, the above three combinatorial geodesics cross distinct hyperplanes. 
Thus, their concatenation is a combinatorial geodesic from $v_i$ to $u_{i+1}$ given by
\[
\gamma_i = [v_{i},\phi^{i}(a)]*[\phi^{i}(a),\phi^i(b)]*[\phi^i(b), u_{i+1}].
\]
As a result, we obtain a pair of combinatorial geodesics $\gamma_{[v_i, u_{i+1}]}$ and $\gamma_i$ with the same endpoints. In particular, they cross the same family of hyperplanes, namely those separating $v_i$ and $u_{i+1}$. Because the combinatorial geodesic $\gamma_{[u_0,v_3]}$ is a concatenation
\[
\gamma_{[u_0,v_3]}
= e_0 * \gamma_{[v_0,u_1]} * e_1 * \gamma_{[v_1,u_2]} * e_2 * \gamma_{[v_2,u_3]} * e_3.
\]
Replacing $\gamma_{[v_i,u_{i+1}]}$ by $\gamma_i$, we obtain a path from $u_0$ to $v_3$
\[
\gamma' = e_0 * \gamma_0 * e_1 * \gamma_1 * e_2 * \gamma_2 * e_3,
\]
which crosses the same family of hyperplanes as $\gamma_{[u_0,v_3]}$.
So $\gamma'$ is also a combinatorial geodesic. In particular, the subsegment $\gamma'_{[a,\phi^2(a)]}$ from $a$ to $\phi^2(a)$ is a combinatorial geodesic. Thus,
\begin{equation}\label{e_1,e_2 are in C}
e_1,e_2 \subset \gamma'_{[a,\phi^2(a)]} \subset \Hull\{a, \phi^2(a)\} = C.
\end{equation}
Similarly, since $a,\phi^2(a)$ are contained in the subsegment $\gamma'_{[v_0, u_3]}$, we have
\begin{equation}\label{a is in hull}
a,\phi^2(a)\in \gamma'_{[v_0, u_3]}\subset \Hull\{e_0,e_3\}.
\end{equation}
Since $C$ and $\Hull\{e_0,e_3\}$ are convex, \eqref{e_1,e_2 are in C} and \eqref{a is in hull} together imply that
\[
\gamma_{[e_1,e_2]} 
\subset C = \Hull\{a, \phi^2(a)\} 
\subset \Hull\{e_0,e_3\}.
\]
This proves (1). We illustrate the above in Figure \ref{figure: subsegment contained in C}(B).

Before proving (2), we make the following observation. Consider the combinatorial geodesic from $a$ to $\phi(a)$ given by
\[
\lambda = \gamma'_{[a,\phi(a)]} = [a,b]*[b,u_1]*e_1*[v_1,\phi(a)].
\]
For each hyperplane crossing $\lambda$, it either separates $H$ and $\phi(H)$ (if it crosses $[a,b]$), or intersects $\phi(H)$ (if it crosses $[b,u_1]*e_1*[v_1,\phi(a)]$). Because $H,\phi(H)$ do not intersect and are not simultaneously crossed by any hyperplane, all hyperplanes crossing $\lambda$ are contained in $H^+$. Consequently, $\phi^{k}(\lambda)$ only crosses hyperplanes contained in $\phi^k(H^+)$. On the other hand, any combinatorial geodesic $[a,\phi^k(a)]$ from $a$ to $\phi^k(a)$ crosses hyperplanes that either intersect $\phi^k(H)$ or are contained in $\phi^k(H^-)$. Therefore, it can always be extended to a longer combinatorial geodesic $[a,\phi^k(a)]* \phi^k(\lambda)$ from $a$ to $\phi^{k+1}(a)$.

Now we prove $\bigcup_{i=0}^n\phi^i(C) = \Hull\{a,\phi^{n+2}(a)\}$ in (2) by induction on $n$. When $n=0$, this holds since $C = \Hull\{a,\phi^{2}(a)\}$. Suppose $\cup_{i=0}^{n-1}\phi^i(C) = \Hull\{a,\phi^{n+1}(a)\}$ for $n\geq 1$. Since any combinatorial geodesic $[a,\phi^{n+1}(a)]$ can be extended to a combinatorial geodesic $[a,\phi^{n+1}(a)]*\phi^{n+1}(\lambda)$ from $a$ to $\phi^{n+1}(a)$, we have 
\begin{equation}\label{first segment}
\Hull\{a,\phi^{n+1}(a)\}\subset \Hull\{a,\phi^{n+2}(a)\}.
\end{equation}
Moreover, because $\bigcup_{k=0}^{n+1}\phi^k(\lambda)$ is a combinatorial geodesic from $a$ to $\phi^{n+2}(a)$ that contains $\phi^{n}(a)$, we have $\phi^{n}(a)\in \Hull\{a,\phi^{n+2}(a)\}$. Thus, 
\begin{equation}\label{second segment}
\phi^n(C) = \Hull\{\phi^n(a), \phi^{n+2}(a)\}\subset \Hull\{a,\phi^{n+2}(a)\}.
\end{equation}
Combining \eqref{first segment} and \eqref{second segment}, we obtain
\begin{equation}\label{equation: geodesic one direction}
\bigcup_{i=0}^n\phi^i(C) 
= \left(\bigcup_{i=0}^{n-1}\phi^i(C)\right) \cup \phi^n(C)
= \Hull\{a,\phi^{n+1}(a)\} \cup \phi^n(C)
\subset \Hull\{a,\phi^{n+2}(a)\}.
\end{equation}
Conversely, exchange $H^+$ and $H^-$ if necessary, we may assume $\phi(H^+)\subset H^+$. So we have nested halfspaces
\[
H^+ \supset \phi(H^+) \supset \phi^2(H^+)\supset \cdots\supset \phi^{n+2}(H^+).
\]
Let $\gamma$ be an arbitrary combinatorial geodesic from $a$ to $\phi^{n+2}(a)$.
Since $a\in H^+\setminus \phi(H^+)$ and $\phi^{n+2}(a)\in \phi^{n+2}(H^+)$, $\gamma$ must cross each $\phi^i(H)$ at an edge $e_i$ for $1\leq i\leq n+2$. Let $e_0$ be the edge dual to $H$ with endpoint $a$. Then $e_0*\gamma$ is a combinatorial geodesic crossing $H, \phi(H),\dots, \phi^{n+2}(H)$ as shown in Figure \ref{figure: extended combinatorial geodesic}. Because $n\geq 1$, $\gamma$ crosses four hyperplanes $\phi^{n-1}(H),\phi^{n}(H),\phi^{n+1}(H),\phi^{n+2}(H)$.
Then, by (1), the subsegment $\gamma_{[e_n,e_{n+1}]}$ is contained in $\phi^{n-1}(C)$. In particular,
\begin{equation}\label{e_n+1}
e_{n+1}\subset \phi^{n-1}(C)= \Hull\{a,\phi^{n+1}(a)\}.
\end{equation}
Alternatively, the extended combinatorial geodesic $\gamma*\phi^{n+2}(\lambda)$ crosses four hyperplanes $\phi(H),\dots, \phi^{n+3}(H)$ (see Figure \ref{figure: extended combinatorial geodesic}). By (1), the subsegment $\gamma_{[e_{n+1},e_{n+2}]}$ is contained in $\phi^{n}(C)$. In particular, $e_{n+1}\subset \phi^{n}(C)$. Combining this and \eqref{e_n+1}, we have
\[
\gamma 
= \gamma_{[a,e_{n+1}]}\cup \gamma_{[e_{n+1}, \phi^{n+2}(a)]}
\subset \Hull\{a,\phi^{n+1}(a)\}\cup \phi^n(C)
= \bigcup_{i=0}^n\phi^i(C).
\]
Since $\gamma$ is an arbitrary combinatorial geodesic from $a$ to $\phi^{n+2}(a)$, this implies that 
\begin{equation}\label{equation: geodesic other direction}
\Hull\{a,\phi^{n+2}(a)\}\subset \bigcup_{i=0}^n\phi^i(C).
\end{equation}
We thus conclude from \eqref{equation: geodesic one direction} and \eqref{equation: geodesic other direction} that $\bigcup_{i=0}^n\phi^i(C) = \Hull\{a,\phi^{n+2}(a)\}$. This finishes the induction and completes the proof of (2).
\end{proof}
\begin{figure}[h]
    \centering
    \includegraphics[width=0.75\linewidth]{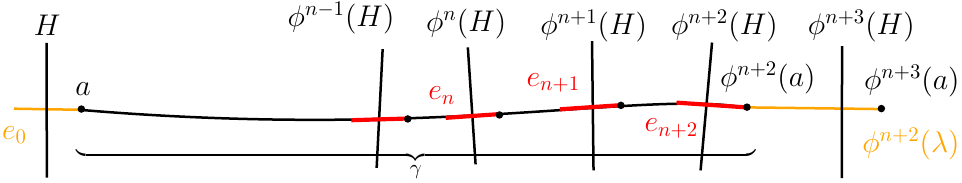}
    \caption{The combinatorial geodesic $\gamma$ can be extended to $e_0*\gamma$ and $\gamma*\phi^{n+1}(\lambda)$.}
    \label{figure: extended combinatorial geodesic}
\end{figure}

\begin{prop}\label{proposition}
Let $X$ be a compact nonpositively curved cube complex and let $(\tY,\phi)$ be a quasiline in $\tX$. 
Then there exist a constant $B_0$ and an integer $n$, which depend only on the quasiline $(\tY,\phi)$, such that for any constant $B\geq B_0$, the following holds:

Let $W$ be a convex subcomplex of $\tY$, and let $\gamma \subset W$ be a combinatorial geodesic that does not cross any trivial hyperplanes of $\tY$.
If $\diam(\gamma)\geq 3B$, then there exist a convex subcomplex $C\subset \tY$ and an integer $l$ such that 
\begin{enumerate}
    \item $\gamma$ contains a subsegment $\hgamma$ with length $\diam(\hgamma)> B$ such that
    \[
    \hgamma\subset C\cup \phi^n(C)\cup \cdots \cup \phi^{nl}(C)\subset W;
    \]
    \item the pair $\left(\bigcup_{k\in \bZ}\phi^{nk}(C), \phi^n\right)$ is a quasiline in $\tX$.
\end{enumerate}
\end{prop}
\begin{rmk}
We say that the convex subcomplex $C$ in $(\tY,\phi)$ \textit{fellow-travels with the combinatorial geodesic} $\gamma$. The intuition behind this is that the union of translates $\bigcup_{k=0}^l \phi^{nk}(C)$ of $C$ is sufficiently thin and long in the following sense:
\begin{itemize}
    \item The subcomplex $\bigcup_{k=0}^l \phi^{nk}(C)$ is sufficiently long and contains a long subsegment $\hgamma$ of $\gamma$ with length $\diam(\hgamma) > B$.
    \item The subcomplex $\bigcup_{k=0}^l \phi^{nk}(C)$ is sufficiently thin so that it is still contained in $W$.
\end{itemize}
\end{rmk}

\begin{proof}[Proof of Proposition \ref{proposition}]
By Lemma \ref{lemma: properties of three types of hyperplanes}, the quasiline $\tY$ has finitely many trivial hyperplanes, say $h$.
Choose $D,K,d$ as in Lemma \ref{lemma: properties of three types of hyperplanes} and $M,N$ as in Lemma \ref{lemma: essential hyperplane crosses long geodesic}.
Let $n = 2d\cdot h!$ and
\[     
B_0 = \max\left\{
K+1, 2nM, 
\llf\frac{2nN+D}{n\lv\phi\rv} + 2\rrf nM
\right\}.
\]
Since constants $D,K,d, M,N$ depend only on the quasiline $(\tY,\phi)$, the constants $B_0,n$ also depend only on $(\tY,\phi)$.
We prove the statements through four steps. The inequality $B\geq B_0\geq K+1$ will be used in Step 1 to find an essential hyperplane $H$. The inequality $B\geq B_0\geq 2nM$ will be used in Step 2 to construct $\hgamma$. Finally, $B\geq B_0\geq \llf\frac{2nN+D}{n\lv\phi\rv} + 2\rrf nM $ will be used in Step 4 to verify the inequality $\diam(\hgamma)>B$.

\medskip

\noindent
\textbf{Step 1.} \textit{Finding an essential hyperplane $H$ crossing $\gamma$}.
Choose an edge $e_0$ in the middle $1/3$ of $\gamma$. Let $H$ be the hyperplane crossing $e_0$. Since $\gamma$ does not intersect trivial hyperplanes, $H$ is either essential or half-essential.
Because $\gamma$ has length $\diam(\gamma)\geq 3B$ and $e_0$ lies in the middle $1/3$, $e_0$ has distance $\geq B$ from two endpoints of $\gamma$. Thus, two halfspaces of $H$ have diameter $\geq B\geq K+1$. By Lemma \ref{lemma: properties of three types of hyperplanes}(2) and our choice of $K$, the hyperplane $H$ is essential.

\medskip

\noindent
\textbf{Step 2.} \textit{Constructing the convex subcomplex $C$ and subsegment $\hgamma$.}
Let $h$ trivial hyperplanes of $\tY$ be $H_1,\cdots, H_h$. 
The action of $\la \phi \ra$ on $\tY$ induces a permutation of them. Thus, $\phi^{2\cdot h!}$ preserves each trivial hyperplane $H_i$ and its halfspaces $H_i^{\pm}$. In particular, each intersection of halfspaces $H_1^{\epsilon_1}\cap \cdots \cap H_h^{\epsilon_h}$ is preserved by $\phi^{2\cdot h!}$, where $H_i^{\epsilon_i}$ denotes a choice of halfspaces $H_i^{\pm}$ of $H_i$.

Because $\gamma$ does not intersect trivial hyperplanes, $\gamma\subset H_1^{\epsilon_1}\cap \cdots \cap H_h^{\epsilon_h}$ for some intersection.
Consider the restriction $\bH = H\cap (H_1^{\epsilon_1}\cap \cdots \cap H_h^{\epsilon_h})$ of $H$ as a hyperplane in $H_1^{\epsilon_1}\cap \cdots \cap H_h^{\epsilon_h}$. Note that $\bH$ is not crossed by any trivial hyperplane. Moreover, since $\phi^{2\cdot h!}$ acts on the subspace $H_1^{\epsilon_1}\cap \cdots \cap H_h^{\epsilon_h}$, the hyperplane $\phi^{2k\cdot h!}(\bH)$ also lies in $H_1^{\epsilon_1}\cap \cdots \cap H_h^{\epsilon_h}$ for any integer $k$. 
Denote by $N(\bH)$ the carrier of $\bH$ in $H_1^{\epsilon_1}\cap \cdots \cap H_h^{\epsilon_h}$. Recall that we take $n = 2d\cdot h!$. In particular, $\phi^{n}(\bH)\subset H_1^{\epsilon_1}\cap \cdots \cap H_h^{\epsilon_h}$.

Recall that $d$ is chosen as in Lemma \ref{lemma: properties of three types of hyperplanes}. Since $n = 2d\cdot h!> d$, $H$ and $\phi^n(H)$ do not intersect and are only simultaneously crossed by trivial hyperplanes. So $\bH$ and $\phi^{n}(\bH)$ do not intersect and are not simultaneously crossed by any hyperplane.
It then follows from Bridge Theorem~\ref{Thm: bridge theorem} that the projection $N(\bH)\pf N(\phi^{n}(\bH)) = \{a\}$ is a singleton.

Let $q\leq 0$ be the minimum integer $k$ for which $\phi^{kn}(H)$ crosses $\gamma$, and $p\geq 0$ be the maximum. We claim that $p,-q\geq 2$. In fact, recall that $e_0$ lies in the middle $1/3$ of $\gamma$, and has distance $\geq B\geq 2(2dM\cdot h!) = 2nM$ from two endpoints of $\gamma$. 
By Lemma~\ref{lemma: essential hyperplane crosses long geodesic} and our choice of $M$, we have $p,-q\geq \llf\frac{B}{nM}\rrf\geq 2$.

Now, let $C$ be the convex subcomplex of $H_1^{\epsilon_1}\cap \cdots \cap H_h^{\epsilon_h}$ defined by
\[
C = \Hull\{\phi^{qn}(a), \phi^{(q+2)n}(a)\}.
\]
For each $q\leq k\leq p$, let $e_k$ be the edge of $\gamma$ that crosses $\phi^{kn}(H)$. Denote by $\gamma_{[e_i, e_j]}$ the subsegment of $\gamma$ from $e_i$ to $e_{j}$.
Then take
\[
\hgamma = \gamma_{[e_{q+1},e_{p-1}]}.
\] 
We illustrate our construction of $C$ and $\hgamma$ in Figure \ref{figure: periodic convex subcomplex and geodesic}.
\begin{figure}[h]
    \centering
    \includegraphics[width=0.75\linewidth]{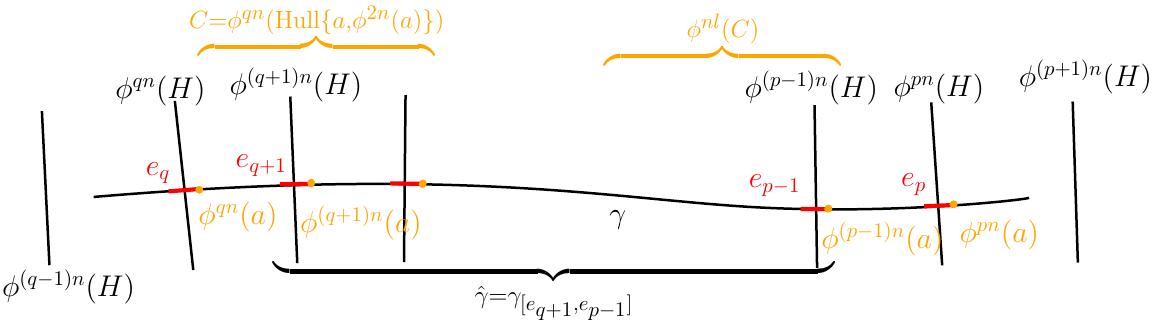}
    \caption{The combinatorial geodesic $\gamma$ crosses hyperplanes $\phi^{qn}(H),\cdots, \phi^{pn}(H)$. The subsegments $\hgamma$ is illustrated by the black brace. The convex subcomplex $C$ and its last translation $\phi^{nl}(C)$ are illustrated by orange braces.}
    \label{figure: periodic convex subcomplex and geodesic}
\end{figure}

\medskip

\noindent
\textbf{Step 3.} \textit{Verifying inclusions and convexity.} 
By Lemma \ref{lemma: property of single convex subcopmlex}(1), for each $q\leq k\leq p-3$, we have
\[
\hgamma_{[e_{k+1}, e_{k+2}]}
\subset \phi^{k-q}(C) = \Hull\{\phi^{kn}(a), \phi^{(k+2)n}(a)\}
\subset \Hull\{e_k, e_{k+3}\}.
\]
Thus,
\[
\hgamma 
= \bigcup_{i=q+1}^{p-2}\hgamma_{[e_{i}, e_{i+1}]}
\subset \bigcup_{k=0}^{p-q-3} \phi^{kn}(C)
\subset \bigcup_{j=q}^{p-3} \Hull\{e_{j}, e_{j+3}\}
\subset \Hull(\gamma)\subset W.
\]
This proves the inclusions in (1) with $l = p-q-3$. 

To prove (2), note that by Lemma~\ref{lemma: property of single convex subcopmlex}(2), the subcomplex $\bigcup_{k=0}^m\phi^{nk}(C)$ is convex for each integer $m$.
Thus, the subcomplex
\[
\bigcup_{k\in \bZ}\phi^{nk}(C)
= \bigcup_{m=1}^\infty\bigcup_{k= -m}^{m}\phi^{nk}(C)
= \bigcup_{m=1}^\infty\left(\phi^{-m}\left(\bigcup_{k= 0}^{2m}\phi^{nk}(C)\right)\right)
\]
is an increasing union of convex subcomplexes, and hence is convex. By our construction, the subcomplex $C$ is compact. So the group $\la\phi^n\ra$ acts cocompactly on $\bigcup_{k\in \bZ}\phi^{nk}(C)$, and the pair $\left(\bigcup_{k\in \bZ}\phi^{nk}(C), \phi^k \right)$ is a quasiline.

\medskip

\noindent
\textbf{Step 4.} \textit{Verifying the inequality.} 
Finally, we prove $\diam(\hgamma)> B$. Because $\diam(\gamma)\geq 3B$, it suffices to show that $\diam(\hgamma)\geq \diam(\gamma)/2$. 

Recall that $D$ is chosen as in Lemma \ref{lemma: properties of three types of hyperplanes}(1).
Because hyperplanes $\phi^{(q+1)n}(H)$ and $\phi^{(p-1)n}(H)$ cross $\gamma$ at edges $e_{q+1}$ and $e_{p-1}$, the length of $\hgamma = \gamma_{[e_{q+1},e_{p-1}]}$ satisfies
\begin{align*}
\diam\left( \hgamma \right)
&\geq ((p-1)n - (q+1)n) \lv\phi\rv - 2\diam(H)\\
&\geq ((p-1) - (q+1)) n\lv\phi\rv - 2D\\
&= (p-q-2)n\lv\phi\rv - 2D.
\end{align*}
On the other hand, because $p,q$ are minimum and maximum integers $k$ for which $\phi^{kn}(H)$ crosses $\gamma$,
the complement of $\hgamma$ in $\gamma$ satisfies (where we exchange $H^+$ and $H^-$ if necessary)
\[
\gamma\setminus \hgamma
\subset (\phi^{(q-1)n}(H^+)\cap \phi^{(q+1)n}(H^-))\cup (\phi^{(p-1)n}(H^+)\cap \phi^{(p+1)n}(H^-)).
\]
By Lemma \ref{lemma: essential hyperplane crosses long geodesic} and our choice of $N$, we have
\begin{align*}
\diam(\gamma\setminus \hgamma)
&\leq \diam\left((\phi^{(q-1)n}(H^+)\cap \phi^{(q+1)n}(H^-))\cup (\phi^{(p-1)n}(H^+)\cap \phi^{(p+1)n}(H^-))\right)\\
&\leq ((q-1) - (q+1))nN + ((p+1)-(p-1))nN\\
&\leq 4nN.
\end{align*}
Thus,
\[
\frac{\diam(\hgamma)}{\diam(\gamma\setminus \hgamma)}
\geq \frac{n(p-q-2)\lv\phi\rv - 2D}{4nN}.
\]
Because $e_0$ has distance $\geq B\geq \llf\frac{2nN+D}{n\lv\phi\rv} + 2\rrf nM$ from two endpoints of $\gamma$, Lemma \ref{lemma: essential hyperplane crosses long geodesic}(2) and our choice of $M$ imply that $p,-q \geq \llf\frac{B}{nM}\rrf\geq \llf\frac{2nN+D}{n\lv\phi\rv} + 2\rrf$. Therefore,
\[
\frac{\diam(\hgamma)}{\diam(\gamma\setminus \hgamma)}
\geq \frac{\left(2\llf\frac{2nN+D}{n\lv\phi\rv} + 2\rrf-2\right)\cdot n\lv\phi\rv - 2D}{4nN}
\geq \frac{\left(2\left(\frac{2nN+D}{n\lv\phi\rv} + 1\right)-2\right)\cdot n\lv\phi\rv - 2D}{4nN}
= 1.
\]
This shows that the subsegment $\hgamma$ contains more than half of edges of $\gamma$. Consequently, 
\[
\diam(\hgamma)
\geq \frac{\diam(\gamma)}{2}
\geq \frac{3B}{2} > B.
\]
This completes the proof.
\end{proof}

\begin{rmk}\label{remark: convex subcomplex C already gives essential hyperplane}
Note that the convex subcomplex $C$ is obtained in Step 2 via an essential hyperplane $H$ of the quasiline $\tY$, where $a\in N(H)\cap C$. Thus, when we apply Proposition \ref{proposition} to a combinatorial geodesic $\gamma$ in a quasiline, we automatically get an essential hyperplane $H$ in $\tY$ crossing $\widehat{\gamma}$, whose carrier $N(H)$ intersects $C$.
\end{rmk}
\begin{rmk}\label{remark: l only depends on B}
The construction of $C$ in Proposition \ref{proposition} depends only on the geometry of quasiline $(\tY,\phi)$, and the number $l$ of its translations mainly depends on $B$. In fact, by Lemma \ref{lemma: essential hyperplane crosses long geodesic} and our construction of integers $l,p,-q$ in the proof, we have
\[
l = p-q-3
\geq 2\llf\frac{B}{nM}\rrf-3
> 2\frac{B}{nM}-5,
\]
where $M$ and $n = 2d\cdot h!$ depend only on the geometry of $(\tY,\phi)$. In particular, we obtain large $l$ once we apply Proposition \ref{proposition} with large $B$.
\end{rmk}

We end this section with the following result that will be used in the proof of Theorem~\ref{theorem: main theorem}.
Recall that a group $G$ satisfies the \textit{unique root property} if for any elements $x,y\in G$, $x^n = y^n$ implies $x = y$ for any nonzero integer $n$. 
Because finitely generated right-angled Artin groups are bi-orderable \cite{DuchampKrob1992}, and all bi-orderable groups satisfy the unique root property \cite[Lemma~6.3]{Minasyan2012}, finitely generated right-angled Artin groups satisfy the unique root property. 
Note that the unique root property passes to subgroups.
Since special groups embed into right-angled Artin groups \cite[Theorem~1.1]{Haglund-Wise2008}, they also satisfy the unique root property.

\begin{lem}\label{lemma: quasi-lines with high bridge}
Suppose that $X$ is a compact virtually special cube complex.
Let $(\tY_1,\phi_1)$ and $(\tY_2,\phi_2)$ be quasilines in $\tX$.
Then there exist constants $S,m$, depending on quasilines $(\tY_1,\phi_1), (\tY_2,\phi_2)$ and $\tX$, such that $\phi_1^{m\lv \phi_2\rv} = \phi_2^{\pm m\lv \phi_1\rv}$ whenever $\diam(\tY_1\pf \tY_2)> S$. 

Moreover, if $\bS\geq S$, and $\bm$ is a multiple of $m$, then $\bS,\bm$ also satisfy the statement.
\end{lem}
\begin{proof}
Because $X$ is compact virtually special, it has a compact special cover $\bX$. Take a compact regular cover $\hX$ of $X$ factoring through $\bX$. Then $\pi_1\hX \triangleleft \pi_1X$ is a normal subgroup of finite index $m = [\pi_1X: \pi_1\hX]$. Moreover, $\hX$ is special as a covering space of the special cube complex $\bX$ \cite[Lemma~3.7]{Haglund-Wise2008}. In particular, the normal subgroup $\pi_1\hX$ satisfies the unique root property.

For each quasiline $(\tY_i,\phi_i)$, the inclusion $\tY_i\hookrightarrow \tX$ induces a local isometry $\tY_i/\la\phi_i\ra\to \tX/\pi_1X = X$. By Corollary \ref{corollary: gates are bounded}, there exists a constant $S$ depending on local isometries $\tY_1/\la\phi_1\ra, \tY_2/\la\phi_2\ra\to X$ such that $\diam(\tY_1\pf \tY_2/\la\phi_1\ra\cap \la\phi_2\ra)\leq S$. Thus, when the projection has diameter $\diam(\tY_1\pf \tY_2)>S$, the intersection $\la\phi_1\ra\cap \la\phi_2\ra$ is nontrivial, and there exist nonzero integers $d_1,d_2$ such that $\phi_1^{d_1} = \phi_2^{d_2}$. In particular, $\lv d_1\rv \lv \phi_1\rv = \lv d_2\rv \lv \phi_2\rv$. Therefore,
\[
\left(\phi_1^{m|\phi_2|}\right)^{|d_1||\phi_1|} = \phi_1^{m|d_1||\phi_1||\phi_2|}
= \phi_2^{\pm m|d_2||\phi_1||\phi_2|} 
= \left(\phi_2^{\pm m|\phi_1|}\right)^{|d_2||\phi_2|}
= \left(\phi_2^{\pm m|\phi_1|}\right)^{|d_1||\phi_1|},
\]
where the choice of $\phi_2$ and $\phi_2^{-1}$ depends on whether $d_1/d_2$ is positive or negative.
Since $m = [\pi_1X: \pi_1\hX]$ is the index of the normal subgroup $\pi_1\hX$, elements $\phi_1^{m|\phi_2|}, \phi_2^{m|\phi_1|}$ are contained in $\pi_1\hX$. Because $\pi_1\hX$ satisfies the unique root property, either $\phi_1^{m|\phi_2|} = \phi_2^{m|\phi_1|}$ or $\phi_1^{m|\phi_2|} = \phi_2^{-m|\phi_1|}$.

Finally, note that the above argument also applies to any constant $\bS\geq S$. Moreover, $\phi_1^{m\lv \phi_2\rv} = \phi_2^{\pm m\lv \phi_1\rv}$ clearly implies $\phi_1^{\bm\lv \phi_2\rv} = \phi_2^{\pm \bm\lv \phi_1\rv}$ for any multiple $\bm$ of $m$. This completes the proof.
\end{proof}

\section{Graph of cube complexes with almost cyclonormal edges}\label{section: main theorem}
\subsection{Almost cyclonormal edges}
Recall that a subgroup $H$ of a group $G$ is \textit{cyclonormal} if for any element $g\in G\setminus H$, the subgroup $H\cap H^g$ is either trivial or cyclic, where $H^g$ denotes the conjugate $g^{-1}Hg$.
The following definition generalizes \cite[Definition~4.5.1]{Wise1996NPCSquareComplexes} to graph of nonpositively curved cube complexes.
\begin{defn}\label{definition: almost cyclonormal edges}
Let $X_\Gamma$ be a graph of nonpositively curved cube complexes. We say that $X_\Gamma$ has \textit{cyclonormal edges} if, for every edge $e$ of $\Gamma$, every endpoint $v$ of $e$, and the associated attaching map $\varphi: X_e\to X_v$, the subgroup $\varphi_*(\pi_1X_e)\leq \pi_1X_v$ is cyclonormal. 

We say that $X_\Gamma$ has \textit{almost cyclonormal edges} if, for every vertex $v$ of $\Gamma$, the following conditions hold:
\begin{enumerate}
    \item For any edge $e$ with endpoint $v$ and the attaching map $\varphi:X_e\to X_v$, if elements $g_1,g_2,g_3\in \pi_1X_v$ represent distinct left cosets of $\varphi_*(\pi_1X_e)$ in $\pi_1X_v$, then the subgroup $\varphi_*(\pi_1X_e)^{g_1}\cap \varphi_*(\pi_1X_e)^{g_2}\cap \varphi_*(\pi_1X_e)^{g_3}$ is either trivial or cyclic.
    \item For any pair of distinct edges $e,f$ with a common endpoint $v$ and the attaching maps $\varphi: \pi_1X_{e}\to \pi_1X_v$, $\psi: \pi_1X_f\to \pi_1X_v$, if elements $g_1,g_2$ represent distinct left cosets of $\psi_*(\pi_1X_f)$ in $\pi_1X_v$, then the subgroup $\varphi_*(\pi_1X_e)\cap \psi_*(\pi_1X_f)^{g_1}\cap \psi_*(\pi_1X_f)^{g_2}$ is either trivial or cyclic.
\end{enumerate}
\end{defn}
\begin{rmk}
Recall from Remark~\ref{remark: simplicial graph} that the graph of nonpositively curved cube complexes $X_\Gamma$ is assumed to be defined over a simplicial graph $\Gamma$. In general, when the underlying graph $\Gamma$ is not simplicial, we should also consider the case $e=f$ in (2).

By definition, a graph of nonpositively curved cube complexes $X_\Gamma$ with cyclonormal edges also has almost cyclonormal edges. 

Because the fundamental group $\pi_1X_v$ of each vertex space is torsion-free, its nontrivial cyclic subgroups are infinite. Thus, the triple intersections in Definition~\ref{definition: almost cyclonormal edges} are either trivial or infinite cyclic.
\end{rmk}
\begin{rmk}\label{remark: triple intersection}
Let $X$ be a cube complex that splits as a graph of nonpositively curved cube complexes $X_\Gamma$.
Consider the action of $\pi_1X$ on the associated Bass--Serre tree $\cT$. Let $\te,\tf$ be a pair of distinct edges in $\cT$ with a common vertex $\tv$. If there exists $g\in \stab(\te)\setminus \stab(\tf)$. Then, regardless of whether $\te,\tf$ are lifts of the same edge in $\Gamma$ or of distinct edges, the three subgroups $\stab(\te), \stab(\tf), \stab(\tf)^{g^{-1}}$ are distinct in $\stab(\tv)$. Therefore, when $X_\Gamma$ has almost cyclonormal edges, the triple intersection 
\[
\stab(\te) \cap \stab(\tf) \cap \stab(\tf)^{g^{-1}}
\]
is either trivial or infinite cyclic.
\end{rmk}
\begin{thm}\label{theorem: main theorem}
Let $X$ be a compact cube complex that splits as a graph of virtually special cube complexes with almost cyclonormal edges. Then $\pi_1X$ has finite stature with respect to the vertex groups.
\end{thm}
By Theorem \ref{theorem: fintie stature equivalent to virtual specialness}, we have the following criterion of virtual specialness.
\begin{cor}\label{corollary: virtually special}
Let $X$ be a compact cube complex that splits as a graph of nonpositively curved cube complexes with (almost) cyclonormal edges. If the vertex groups are hyperbolic, then $X$ is virtually special.
\end{cor}

\begin{proof}[Proof of Theorem \ref{theorem: main theorem}]
Suppose $X$ splits as a graph of nonpositively curved cube complexes $X_\Gamma$ with almost cyclonormal edges.
Let $\cT$ be the associated Bass--Serre tree. For each vertex $v$ of $\Gamma$, let $\tv$ be the lift of $v$ in $\cT$ with $\stab(\tv) = \pi_1X_v$. 
Let $\rho:[0,n]\to \cT$ be an arbitrary finite path starting from the vertex $\rho(0) =  \tv$. 
Recall from the proof of Lemma~\ref{lemma: represent path stabilizer} that in the universal cover $\tX$ of $X$, the finite path $\rho$ corresponds to a subcomplex
\[
\bigsqcup_{i=0}^n \tX_i \sqcup \left(\bigsqcup_{j=1}^n\tY_j\times[0,1]\right)/\sim 
\]
where $X_0 = X_v, X_1,\dots, X_n$ and $Y_1,\dots, Y_n$ are (not necessarily distinct) vertex spaces and edge spaces.
For each $j$, $\tY_j\times \{0\}$ is attached to an elevation $\tY_j^0\subset \tX_{i-1}$ of $Y_j$ and $\tY_j\times \{1\}$ is attached to an elevation $\tY_j^1\subset \tX_i$. 
For simplicity, we use $\tY_j^i$ to denote $\tY_j\times \{i\}$ with $i = 0,1$.
Consider the projection 
\[
\tY_1^0 \pf \tY_n^0
= \tY_1^0 \pf \tY_2^0 \pf \cdots \pf \tY_n^0 \subset \tY_1^0\subset \tX_0.
\]
By Lemma \ref{lemma: represent path stabilizer}, the local isometry 
\[
(\tY_1^0 \pf \tY_n^0)/\pstab(\rho)\longrightarrow X_0 = X_v
\]
represents the subgroup $\pstab(\rho)\leq \pi_1X_v$. We claim that there exists a constant $P$ such that, for each such finite path $\rho$ with infinite pointwise stabilizer $\pstab(\rho)$, we have 
\begin{equation}\label{equation: upper bounda of diameter in main thm}
    \diam\left( (\tY_1^0 \pf \tY_n^0)/\pstab(\rho) \right) \leq P.
\end{equation}
In particular, this implies that there are finitely many local isometries of the form $(\tY_1^0 \pf \tY_n^0)/\pstab(\rho)\to X_v$, which induce finitely many $\pi_1X_v$-conjugacy classes of infinite subgroups of $\pi_1X_v$ of the form $\pstab(\rho)$. It then follows from Lemma \ref{lemma: vefity finite stature on paths} that $\pi_1X$ has finite stature with respect to the vertex groups. 

We prove \eqref{equation: upper bounda of diameter in main thm} via the following steps:
\begin{itemize}
    \item In Step 1, we find a large enough constant $P$.
    \item In Step 2, we start with an arbitrary combinatorial geodesic $\gamma$ of length $P$ in $\tY_1^0\pf \tY_n^0$, and use it to construcct a quasiline $(L_k, \phi_k)$ in $\tY_k^1$ for a specific $k$. Moreover, $L_k$ contains a long combinatorial geodesic $\gamma_b^k$.
    \item In Step 3, we apply Proposition~\ref{proposition} to a subsegment of $\gamma_b^k$ in the quasiline $L_k$, and obtain a fellow-traveling convex subcopmlex $C_k$.
    \item In Steps 4--5, we use the element $\phi_k$ to construct an element $\phi$, and show that $\phi\in \pstab(\rho)$ by contradiction.
    \item In Step 6, we show that $\alpha\cdot \gamma\cap \gamma\neq \emptyset$ by taking $\alpha = \phi^d$ for some power $d$. This implies \eqref{equation: upper bounda of diameter in main thm}.
\end{itemize}

\medskip

\noindent
\textbf{Step 1.} \textit{Finding the upper bound $P$ of diameter.} We consider the following constants that depend only on $X$:
\begin{enumerate}[label=(\alph*)]
\item \label{item: choice of S}
Suppose that $e,f$ are a pair of (not necessarily distinct) edges in $\Gamma$ with a common vertex $v$. Let attaching maps of their edge spaces be $\varphi: X_e\to X_v$ and $\psi: X_f\to X_v$. 
By Corollary~\ref{corollary: gates are bounded}, there exists a constant $S_{e,f,v}$ such that 
\[
\diam\left(g_1\tX_{e}\pf g_2\tX_{f}\pf g_3\tX_{f}/ 
\varphi_*(\pi_1X_{e})^{g_1^{-1}}\cap {\psi_*(\pi_1X_{f})}^{g_2^{-1}}\cap {\psi_*(\pi_1X_{f})}^{g_3^{-1}}\right)\leq S_{e,f,v}
\]
for any $g_1,g_2,g_3\in \pi_1X_v$. Since $X$ is compact, $\Gamma$ is a finite graph.
So there are finitely many such constants $S_{e,f,v}$. Let
\[
S = \max\{S_{e,f,v}\mid e,f\in E(\Gamma),\ \text{$v$ is a common vertex of $e,f$}\}.
\]

\item \label{item: representative of quasilines and trivial hyperplane}
Suppose again that $e,f$ are a pair of (not necessarily distinct) edges in $\Gamma$ with a common vertex $v$. Let attaching maps of their edge spaces be $\varphi: X_e\to X_v$ and $\psi: X_f\to X_v$. Recall that edges of $X_\Gamma$ are almost cyclonormal. For any $g_1,g_2,g_3\in \pi_1X$, suppose that one of the following holds:
\begin{itemize}
    \item $e = f$ and $g_1,g_2,g_3$ represent distinct left cosets of $\varphi_*(\pi_1X_e)$ in $\pi_1X_v$; 
    \item $e \neq f$ and $g_2,g_3$ represent distinct left cosets of $\psi_*(\pi_1X_f)$ in $\pi_1X_v$.
\end{itemize}
Then the subgroup $\varphi_*(\pi_1X_{e})^{g_1^{-1}}\cap {\psi_*(\pi_1X_{f})}^{g_2^{-1}}\cap {\psi_*(\pi_1X_{f})}^{g_3^{-1}}$ is either trivial or infinite cyclic. In the cyclic case, the complex 
\[
g_1\tX_{e}\pf g_2\tX_{f}\pf g_3\tX_{f}/ 
\varphi_*(\pi_1X_{e})^{g_1^{-1}}\cap {\psi_*(\pi_1X_{f})}^{g_2^{-1}}\cap {\psi_*(\pi_1X_{f})}^{g_3^{-1}}
\]
lifts to a quasiline 
\[
\left(
L = g_1\tX_{e}\pf g_2\tX_{f}\pf g_3\tX_{f}, \phi
\right)
\]
where $\phi$ is a generator of $\varphi_*(\pi_1X_{e})^{g_1^{-1}}\cap {\psi_*(\pi_1X_{f})}^{g_2^{-1}}\cap {\psi_*(\pi_1X_{f})}^{g_3^{-1}}$. By Corollary~\ref{corollary: gates are bounded} and our choice of $S$ in \ref{item: choice of S}, we have 
\[
\diam\left(L/ 
\varphi_*(\pi_1X_{e})^{g_1^{-1}}\cap {\psi_*(\pi_1X_{f})}^{g_2^{-1}}\cap {\psi_*(\pi_1X_{f})}^{g_3^{-1}}\right)\leq S.
\]
Thus, there are finitely many isomorphism classes of such complexes among all pair of (not necessarily distinct) edges $e,f\in E(\Gamma)$ with the common vertex $v$. They induce finitely many $\pi_1X_v$-conjugacy classes of subgroups of the form
\[
\varphi_*(\pi_1X_{e})^{g_1^{-1}}\cap {\psi_*(\pi_1X_{f})}^{g_2^{-1}}\cap {\psi_*(\pi_1X_{f})}^{g_3^{-1}}\leq \pi_1X_v
\]
for each vertex $v\in V(\Gamma)$.
As a result, for each $v\in V(\Gamma)$, there are finitely many isomorphism classes of quasilines of the form $\left(g_1\tX_{e}\pf g_2\tX_{f}\pf g_3\tX_{f}, \phi\right)$ in $\tX_v$. By collecting all isomorphism classes of quasilines in $\tX_v$ as $v$ ranges over the vertices of $\Gamma$, we obtain finitely many isomorphism classes, with representatives
\[
(L_1', \phi_1'),\dots, (L_s',\phi_s').
\]
Suppose $L_i'$ has $h_i$ trivial hyperplanes. Take $h = \max\{h_i\mid 1\leq i\leq s\}$.

\item \label{type2quasiline}
Apply Proposition~\ref{proposition} to $(L_i',\phi_i')$ for a long combinatorial geodesic, we obtain a quasiline of the form
\[
\left( \bigcup_{p\in \bZ}(\phi_i')^{pn_i}(C_i), (\phi_i')^{n_i} \right)
\]
for some fellow-travelling convex subcomplex $C_i$.
By our construction in Proposition~\ref{proposition}, each quasiline representative $(L_i',\phi_i')$ in \ref{item: representative of quasilines and trivial hyperplane} yields finitely many isomorphism classes of fellow-traveling subcomplexes $C_i$, whose translation under $(\phi_i')^{n_i}$ give finitely many isomorphism classes of quasilines with representatives
\[
\left(J_1,\xi_1\right),\cdots, (J_r,\xi_r).
\]
By applying Lemma~\ref{lemma: quasi-lines with high bridge} to the pair consisting of a representative $(J_i,\xi_i)$ above and a representative $(L_j',\phi_j')$ in \ref{item: representative of quasilines and trivial hyperplane}, we obtain a constant $Q_{ij}$ and an integer $m_{ij}$ such that $\xi_i^{m_{ij}|\phi_j'|} = (\phi_j')^{\pm m_{ij}|\xi_i|}$ whenever $\diam(J_i\pf L_j')\geq Q_{ij}$. Then let $Q$ be the maximum
\[
Q = \max \{Q_{ij}\mid 1\leq i\leq r,\ 1\leq j\leq s\}
\]
and let $m$ the the least common multiple
\[
m = LCM \{m_{ij}\mid 1\leq i\leq r,\ 1\leq j\leq s\}.
\]

\item \label{item: possibilites to cross a hyperplane}
By Lemma \ref{lemma: properties of three types of hyperplanes}, for each quasiline representative $(L_i',\phi_i')$ in \ref{item: representative of quasilines and trivial hyperplane}, there is an upper bound on the diameters of all essential hyperplanes in $L_i'$. Consequently, for each essential hyperplane $H$ of $L_i'$, its number of vertices $\#V(H)$ is finite. Let
\[
\cL = \max_{1\leq i\leq s} \{ \#V(H)\mid  \text{$H$ is an essential hyperplane in $L_i'$}\}.
\]
In particular, there are at most $\cL$ possible ways for combinatorial geodesics to cross each essential hyperplane $H$ of $L_i'$.

\item  \label{item: choice of B}
For each representative $(L_i', \phi_i')$ of a quailine isomorphism class in \ref{item: representative of quasilines and trivial hyperplane}, Lemma~\ref{lemma: essential hyperplane crosses long geodesic} and Proposition~\ref{proposition} yield constants $M_i'$ and $B_i', n_i'$ respectively. Take $M = \max_{1\leq i\leq s} M_i'$. Then let
\[
B = \max_{1\leq i\leq s} \left\{  B_i', \frac{\left(2m\cdot |\phi_1'|\cdots |\phi_s'| \cdot \cL + 5\right)n_i'M}{2}, 2S, 2Q\right\}.
\]
In particular, $B \geq B_i'$ for each $i$, $B\geq2S$, and $B\geq2Q$. Further, if we apply Proposition~\ref{proposition} to $(L_i',\phi_i')$ for a combinatorial geodesic $\gamma$ of length $\geq 3B$, and obtain the subcomplex $C_i\cup \phi_i^n(C_i)\cup \cdots\cup\phi_i^{n_il_i}(C_i)$, by Remark~\ref{remark: l only depends on B},
\[
l_i' \geq 2\frac{B}{n_iM_i} - 5 
\geq 2m\cdot |\phi_1'|\cdots |\phi_s'| \cdot \cL.
\]

\item \label{item: choice of R}
Let $R = 3B(h+1)+h$.
In particular, $R > B> S$ and $ \frac{R-h}{h+1} \geq 3B $.

\item \label{item: maximal combination of edges}
Finally, let $T = \max \{ (\#E(X_e))^R\mid  e\in E(\Gamma)\}$.
In particular, for each edge space $X_e$, there are at most $T$ possible combinatorial paths in $X_e$ of length $R$.
\end{enumerate}
We claim that the constant $P = (T+1)R$ satisfies \eqref{equation: upper bounda of diameter in main thm}, i.e. for any finite path $\rho$ in the Bass--Serre tree with infinite $\pstab(\rho)$, we have $\diam\left( (\tY_1^0 \pf \tY_n^0)/\pstab(\rho) \right) \leq P$.
We prove this in Steps 2--7.

\medskip

\noindent
\textbf{Step 2.} \textit{Constructing the first quasiline.}
Let $\rho: [0,n]\to \cT$ be a finite path with infinite $\pstab(\rho)$.
To prove \eqref{equation: upper bounda of diameter in main thm}, it suffices to show that for each combinatorial geodesic $\gamma\subset \tY_1^0 \pf \tY_n^0$ of length $P$, there exists $\alpha\in \pstab(\rho)$ such that $\alpha\cdot \gamma\cap \gamma\neq \emptyset$.

Let $\gamma\subset \tY_1^0 \pf \tY_n^0\subset \tY_1^0$ be an arbitrary combinatorial geodesic of length $P = (T+1)R$. Decompose it as the concatenation of $T+1$ combinatorial geodesics $\gamma = \gamma_0*\gamma_1*\dots*\gamma_T$, where each $\gamma_i$ has length $R$. By our choice of $T$ in \ref{item: maximal combination of edges}, two of them, say $\gamma_a$ and $\gamma_b$ with $a\neq b$, have the same image in $Y_1$. Hence, there exists nontrivial $g\in \stab(\tY_1^0) = \stab(\rho[0,1])$ such that $g\cdot \gamma_a = \gamma_b$. 

If $g\in \pstab(\rho)$, then $\gamma_b = g\cdot \gamma_a\subset g\cdot \gamma\cap \gamma\neq \emptyset$ and $g$ is the desired element $\alpha$.
Suppose $g\not\in \pstab(\rho)$, then there exists a minimum $k\geq 1$ such that $g\in \pstab(\rho[0,k])\setminus \pstab(\rho[k,k+1])$. In particular,
\[
g\in \stab(\rho[k-1,k])\setminus \pstab(\rho[k,k+1]) 
= \stab(\tY_{k}^1)\setminus \stab(\tY_{k+1}^0).
\]
Then the three subgroups $\stab(\tY_{k}^1), \stab(\tY_{k+1}^0), \stab(\tY_{k+1}^0)^{g^{-1}}$ in $\stab(\tX_k)$ are distinct. Because the edges are almost cyclonormal, by Remark \ref{remark: triple intersection}, the intersection $\stab(\tY_{k}^1)\cap \stab(\tY_{k+1}^0)\cap \stab(\tY_{k+1}^0)^{g^{-1}}$ is either trivial or infinite cyclic. When it is infinite cyclic, by taking a generator $\phi_k$ of $\stab(\tY_{k}^1)\cap \stab(\tY_{k+1}^0)\cap \stab(\tY_{k+1}^0)^{g^{-1}}$, we obtain a quasiline 
\[
\left(L_{k} = \tY_{k}^1\pf \tY_{k+1}^0\pf g\tY_{k+1}^0 , \phi_{k}\right)
\]
in $\tY_{k}^1\subset \tX_{k}$ as shown in Figure \ref{figure: first quasiline}.

We claim that the subgroup $\stab(\tY_{k}^1)\cap \stab(\tY_{k+1}^0)\cap \stab(\tY_{k+1}^0)^{g^{-1}}$ must be infinite cyclic.
In fact, recall that $\gamma_b = g\cdot \gamma_a\subset \gamma\cap g\cdot \gamma$.
Because $\gamma\subset \tY_1^0\pf \tY_n^0= \tY_1^0\pf \tY_{k}^1\pf \tY_n^0$ and $g\in \pstab(\rho[0,k]) = \stab(\tY_1^0\pf \tY_k^1)$, we have
\begin{align*}
\gamma_b
\subset \gamma \cap g\cdot \gamma
&\subset (\tY_1^0\pf \tY_{k}^1\pf \tY_n^0) \cap g(\tY_1^0\pf \tY_{k}^1\pf \tY_n^0)\\
&\subset (\tY_1^0\pf \tY_{k}^1\pf \tY_{k+1}^0) \cap g(\tY_1^0\pf \tY_{k}^1\pf \tY_{k+1}^0)\\
&= (\tY_1^0\pf \tY_{k}^1\pf \tY_{k+1}^0) \cap (\tY_1^0\pf \tY_{k}^1\pf g\tY_{k+1}^0)\\
&\subset \tY_1^0\pf \tY_{k}^1\pf \tY_{k+1}^0\pf g\tY_{k+1}^0\\
&= \tY_1^0\pf L_k.
\end{align*}
Combining this with the inclusion $\gamma_b\subset \gamma$, we have
\[
\gamma_b
\subset (\tY_1^0\pf L_k)\cap \gamma
\subset (\tY_1^0\pf L_k)\cap (\tY_1^0\pf \tY_{k}^1\pf \tY_n^0)
\subset \tY_1^0\pf L_k\pf \tY_{k}^1\pf \tY_n^0.
\]
By Bridge Theorem \ref{Thm: bridge theorem}, $\gamma_b$ is isometric to its projection
\begin{equation}\label{equation: gamma in W}
\gamma_b^k
:= \Pi_{L_k\pf \tY_{k}^1\pf \tY_n^0}(\gamma_b)
\subset L_k\pf \tY_{k}^1\pf \tY_n^0.
\end{equation}
In particular,
\[
\diam(L_k)
\geq \diam(L_k\pf \tY_{k}^1\pf \tY_n^0)
\geq \diam\left(\gamma_b^k \right) 
= \diam(\gamma_b)
= R > S.
\]
Other the other hand, by Corollary \ref{corollary: gates are bounded} and our choice of $S$ in \ref{item: choice of S}, we have 
\[
\diam \left(L_k / 
\stab(\tY_{k}^1)\cap \stab(\tY_{k+1}^0)\cap \stab(\tY_{k+1}^0)^{g^{-1}}  \right) 
\leq S.
\]
This shows that the subgroup $\stab(\tY_{k}^1)\cap \stab(\tY_{k+1}^0)\cap \stab(\tY_{k+1}^0)^{g^{-1}}$ is nontrivial. 
\begin{figure}[h]
    \centering
    \includegraphics[width=0.75\linewidth]{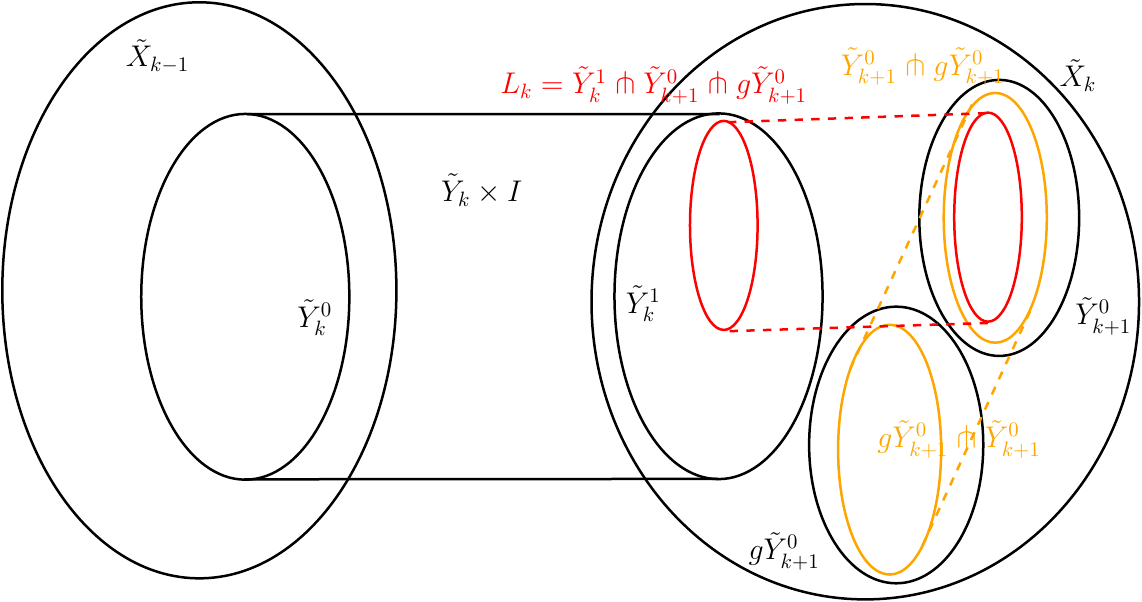}
    \caption{The quasiline $L_k = \tY_k^1 \pf \tY_{k+1}^0 \pf g\tY_{k+1}^0$ is shown as the red region in $\tY_k^1$. The dashed lines illustrate the bridge structures of projections.}
    \label{figure: first quasiline}
\end{figure}

\medskip

\noindent
\textbf{Step 3.} \textit{Constructing a fellow-traveling convex subcomplex $C_k$ in the quasiline.}
Recall from~\eqref{equation: gamma in W} that $\gamma_b^k\subset L_k\pf \tY_k^1\pf \tY_n^0$. In particular, $\gamma_b^k$ is a combinatorial geodesic in the quasiline $(L_k,\phi_k)$ of length 
\[
\diam(\gamma_b^k) = \diam(\gamma_b) = R= 3B(h+1)+h.
\]
By our choice of $h$ in \ref{item: representative of quasilines and trivial hyperplane}, the quasiline $L_{k}$ has $\leq h$ trivial hyperplanes. Because $\gamma_b^k$ crosses each hyperplane at most once, it contains a subsegment $\bar{\gamma}_b^k$ of length $\geq \frac{R-h}{h+1} \geq 3B$ that does not cross any trivial hyperplane of $L_{k}$. Apply Proposition~\ref{proposition} to the combinatorial geodesic $\bar{\gamma}_b^k$ in the quasiline $(L_{k}, \phi_{k})$ with the convex subcomplex $W = L_{k}\pf (\tY_k^1\pf \tY_n^0)$, we obtain a fellow-traveling convex subcomplex $C_{k}\subset L_{k}$ and integers $n_k, l_{k}$ such that 
\begin{itemize}
    \item $\bar{\gamma}_b^k$ contains a subsegment $\hgamma_b^k$ of length $\diam(\hgamma_b^k)> B$ such that 
    \begin{equation}\label{c_k1}
    \hgamma_b^k
    \subset C_{k}\cup \phi_k^{n_{k}}(C_{k})\cup \cdots \cup \phi_{{k}}^{n_{k}l_k}(C_k)
    \subset L_{k}\pf (\tY_k^1\pf \tY_n^0);
    \end{equation}
    \item the pair $\left(\bigcup_{p\in \bZ}\phi_k^{pn_k}(C_k), \phi_k^{n_k}\right)$ is a quasiline.
\end{itemize} 
Moreover, by our choice of $B$ in \ref{item: choice of B}, we have 
\begin{equation}\label{l_k}
l_k 
\geq 2m\cdot |\phi_1'|\cdots |\phi_s'| \cdot \cL
\geq 2m\cdot |\phi_1'|\cdots |\phi_s'|.
\end{equation}
In particular, 
\begin{equation}\label{bound}
\diam\left(\bigcup_{p=m\cdot |\phi_1'|\cdots |\phi_s'|}^{l_k}\phi_{k}^{pn_k}(C_k)\right)
\geq \frac{1}{2}\diam\left(\bigcup_{p=0}^{l_k}\phi_{k}^{pn_k}(C_k)\right)
\geq \frac{1}{2}\diam(\hgamma_b^k)
> B/2\geq S.
\end{equation}

\medskip

\noindent
\textbf{Step 4.} \textit{Constructing an element $\phi$ in $\pstab(\rho)$.}
We claim that $\phi = \phi_k^{n_k\beta}\in \pstab(\rho)$, with the integer $\beta = m\cdot |\phi_1'|\cdots |\phi_s'|$. We prove this in Steps 4--5.

Assume $\phi_k^{n_k\beta}\notin \pstab(\rho)$. 
Recall that since $\phi_k\in \stab(L_k)\leq \pstab(\rho[k-1,k+1])$, there exists an integer $t\neq k$ such that $\phi_k^{n_k\beta}\notin \stab(\rho[t,t+1])$. Without loss of generality, assume that $t > k$, and the case $t < k$ is handled similarly by symmetry.

Let $t\geq k+1$ be the minimal integer such that 
\begin{equation}\label{assume}
\phi_k^{n_k\beta}
\in \stab(\rho[k-1,t])\setminus \stab(\rho[t,t+1])
\leq \stab(\tY_{t}^1)\setminus \stab(\tY_{t+1}^0).
\end{equation}
Then the three subgroups $\stab(\tY_{t}^1), \stab(\tY_{t+1}^0), \stab(\tY_{t+1}^0)^{\phi_k^{-n_k\beta}}$ in $\stab(\tX_t)$ are distinct.  Because the edges are almost cyclonormal, by Remark \ref{remark: triple intersection}, the subgroup $\stab(\tY_{t}^1) \cap \stab(\tY_{t+1}^0) \cap \stab(\tY_{t+1}^0)^{\phi_k^{-n_k\beta}}$ is either trivial or infinite cyclic. When it is infinite cyclic, by taking a generator $\phi_{t}$ of $\stab(\tY_{t}^1) \cap \stab(\tY_{t+1}^0) \cap \stab(\tY_{t+1}^0)^{\phi_k^{-n_k\beta}}$, we obtain a quasiline 
\[
\left(
L_t = \tY_t^1\pf \tY_{t+1}^0\pf \phi_k^{n_k\beta}(\tY_{t+1}^0), \phi_t
\right)
\]
in $\tY_t^1$ as shown in Figure \ref{figure: L_t}.

We claim that the subgroup $\stab(\tY_{t}^1) \cap \stab(\tY_{t+1}^0) \cap \stab(\tY_{t+1}^0)^{\phi_k^{-n_k\beta}}$ must be infinite cyclic. To see this, notice that \eqref{c_k1} gives
\begin{equation}\label{c_k2}
C_k\cup \phi_k^{n_{k}}(C_{k})\cup \cdots \cup \phi_{{k}}^{n_{k}l_k}(C_k)
\subset L_{k}\pf (\tY_k^1\pf \tY_n^0)
= L_{k}\pf (\tY_t^1\pf \tY_n^0).
\end{equation}
Since $\phi_k^{n_k\beta}\in \stab(\rho[k-1,t])\cap \stab(L_k) \leq \stab(L_k\pf \tY_{t}^1)$, we have 
\begin{equation}\label{c_k3}
\begin{aligned}
\bigcup_{p=\beta}^{l_k}\phi^{pn_k}(C_k)
&\subset \left(C_k\cup \cdots \cup \phi_k^{n_kl_k}(C_k)\right) 
\cap \phi_k^{n_k\beta}\left(C_k\cup \cdots \cup \phi_k^{n_kl_k}(C_k)\right)\\
&\subset \left(L_{k}\pf (\tY_t^1\pf \tY_{n}^0)\right)\cap \phi_k^{n_k\beta}\left(L_{k}\pf (\tY_t^1\pf \tY_{n}^0)\right)\\
&\subset \left(L_{k}\pf (\tY_t^1\pf \tY_{t+1}^0)\right)\cap \phi_k^{n_k\beta}\left(L_{k}\pf (\tY_t^1\pf \tY_{t+1}^0)\right)\\
&= \left(L_{k}\pf \tY_t^1\pf \tY_{t+1}^0\right)\cap \left(L_{k}\pf \tY_t^1\pf \phi_k^{n_k\beta}(\tY_{t+1}^0)\right)\\
&\subset L_{k}\pf \left(\tY_t^1\pf \tY_{t+1}^0\pf \phi_k^{n_k\beta}(\tY_{t+1}^0)\right)\\
&= L_k\pf L_t.
\end{aligned}
\end{equation}
Combining this and \eqref{c_k2}, we have 
\begin{equation}\label{c_k4}
\bigcup_{p=\beta}^{l_k}\phi^{pn_k}(C_k)
\subset (L_k\pf L_t) \cap (L_{k}\pf \tY_t^1\pf \tY_n^0)
\subset L_k\pf L_t\pf \tY_t^1\pf \tY_n^0.
\end{equation}
Recall from \eqref{bound} that $\diam\left(\bigcup_{p=\beta}^{l_k}\phi^{pn_k}(C_k)\right)>S$. Thus, 
\[
\diam(L_t)
\geq \diam(L_k\pf L_t\pf \tY_t^1\pf \tY_n^0)
\geq \diam\left(\bigcup_{p=\beta}^{l_k}\phi^{pn_k}(C_k)\right)
>S.
\]
Other the other hand, by Corollary \ref{corollary: gates are bounded} and our choice of $S$ in \ref{item: choice of S}, we have 
\[
\diam \left(L_t / 
\stab(\tY_{t}^1) \cap \stab(\tY_{t+1}^0) \cap \stab(\tY_{t+1}^0)^{\phi_k^{-n_k\beta}} \right) 
\leq S.
\]
This shows that the subgroup $\stab(\tY_{t}^1) \cap \stab(\tY_{t+1}^0) \cap \stab(\tY_{t+1}^0)^{\phi_k^{-n_k\beta}}$ is nontrivial. 
\begin{figure}
    \centering
    \includegraphics[width=0.75\linewidth]{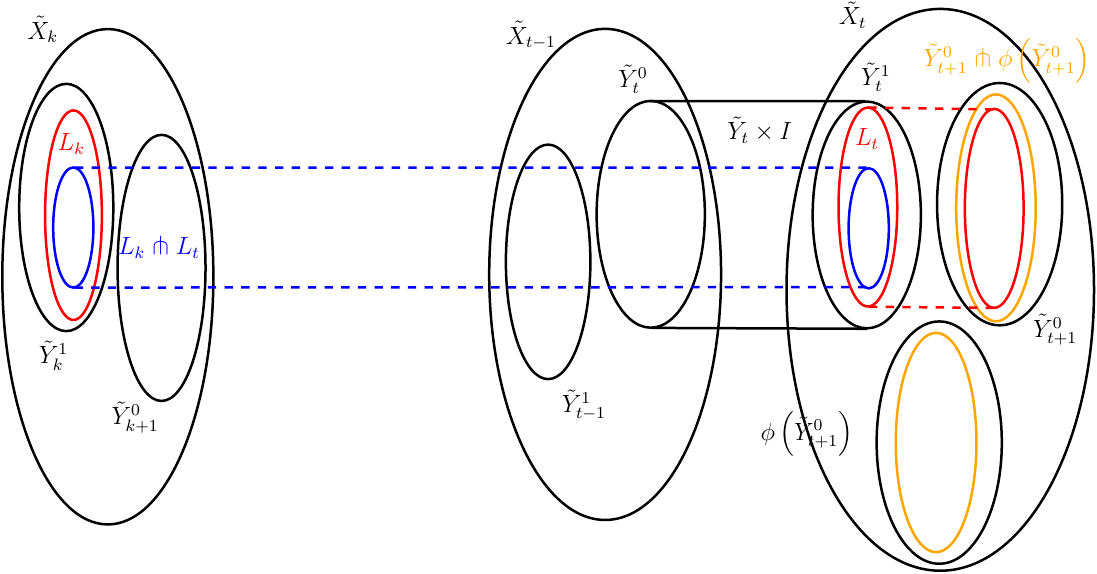}
    \caption{Two quasilines $L_k$ and $L_t$ are shown as the red regions. The blue region is the bridge of their projections $L_k \pf L_t \cong L_t \pf L_k$.}
    \label{figure: L_t}
\end{figure}

\medskip

\noindent
\textbf{Step 5.} \textit{Proving $\phi\in \pstab(\rho)$ by contradiction.} 
Based on the construction of the quasiline $(L_t,\phi_t)$ in Step 4, we claim that $(\phi_k^{n_k})^{m\lv \phi_{t}\rv} = \phi_{t}^{\pm mn_k\lv \phi_k\rv}$. Since $(L_t,\phi_t)$ belongs to an isomorphism class of quasilines in \ref{item: representative of quasilines and trivial hyperplane}, there exists $i_t$ such that $(L_t,\phi_t)$ is isomorphic to the representative $(L_{i_t}', \phi_{i_t}')$. Then the element
\[
\phi
= \phi_k^{n_k\beta}
= \phi_k^{n_k\cdot m\cdot |\phi_1'|\cdots |\phi_s'|}
= \phi_t^{\pm n_k\cdot m\cdot |\phi_1'|\cdots \widehat{|\phi_{i_t}'|}\cdots |\phi_s'|\cdot |\phi_k|}
\]
is a power of $\phi_t$ and hence is an element of $\stab(\rho[t,t+1])$. This contradicts the assumption in \eqref{assume}. Therefore, we must have $\phi = \phi_k^{n_k\beta}\in \pstab(\rho)$.

Now, we prove the claim $(\phi_k^{n_k})^{m\lv \phi_{t}\rv} = \phi_{t}^{\pm mn_k\lv \phi_k\rv}$.
The idea is that we are going to compare two quasilines in the same vertex space $\tX_{t}$. To do this, recall from~\eqref{c_k4} that
\begin{equation}\label{c_k5}
\bigcup_{p=\beta}^{l_k}\phi^{pn_k}(C_k)
\subset L_{k}\pf L_t \pf \tY_t^1\pf \tY_n^0
\subset L_{k}\pf L_t
\subset L_k\pf \tY_t^1.
\end{equation}
By \eqref{l_k}, $l_k \geq 2m\cdot |\phi_1'|\cdots |\phi_s'| \cdot \cL\geq 2\beta$. Thus,
\[
\bigcup_{p=\beta}^{2\beta}\phi^{pn_k}(C_k)
\subset\bigcup_{p=\beta}^{l_k}\phi^{pn_k}(C_k)
\subset L_k\pf \tY_t^1.
\]
Because $\phi_k^{n_k\beta}\in \stab(L_k)\cap \stab(\rho[k-1,t])\leq \stab(L_k\pf \tY_t^1)$, we have
\[
\bigcup_{p\in \bZ}\phi^{pn_k}(C_k)
= \bigcup_{p\in\bZ}\phi_k^{pn_k\beta}\left(\bigcup_{p=\beta}^{2\beta}\phi^{pn_k}(C_k)\right)
\subset L_k\pf \tY_t^1.
\]
By taking its projection to $\tY_t^1$, we obtain a quasiline in $\tY_t^1\subset \tX_t$ given by
\[
\left(J_t = \bigcup_{p\in \bZ}\phi^{pn_k}(C_t), \phi_k^{n_k} \right),
\]
where $C_t = \Pi_{\tY_t^1}(C_k)$.
We now compare two quasilines $(J_t,\phi_k^{n_k})$ and $(L_t, \phi_t)$ in $\tX_t$. 

By our construction, we have $\bigcup_{p=\beta}^{l_k}\phi^{pn_k}(C_k)
\subset L_k \pf \left(\bigcup_{p\in \bZ}\phi^{pn_k}(C_t)\right) = L_k\pf J_t$.
Combining this and~\eqref{c_k5}, we obtain
\[
\bigcup_{p=\beta}^{l_k}\phi^{pn_k}(C_k)
\subset \left(L_k \pf J_t\right)
\cap (L_k\pf L_t)
\subset L_k \pf J_t\pf L_t.
\]
It then follows from \eqref{bound} and our choice of $B$ in \ref{item: choice of B} that
\[
\diam\left(J_t\pf L_t\right)
\geq \diam\left(\bigcup_{p=\beta}^{l_k}\phi^{pn_k}(C_k)\right)
> \frac{B}{2}\geq Q.
\]
Note that $(L_t,\phi_t)$ and $(J_t,\phi_k^{n_k})$ are quasilines constructed in \ref{item: representative of quasilines and trivial hyperplane} and \ref{type2quasiline} respectively.
Apply Lemma~\ref{lemma: quasi-lines with high bridge} to two quasilines $\left(J_t, \phi_k^{n_k} \right)$ and $(L_t,\phi_t)$. By our choice of $Q$ and $m$ in~\ref{type2quasiline}, we have 
\[
\left(\phi_k^{n_k}\right)^{m\lv \phi_{t}\rv} 
= \phi_{t}^{\pm m\left\lv \phi_k^{n_k}\right\rv}
= \phi_{t}^{\pm mn_k\lv \phi_k\rv}.
\]
This completes the proof of the claim.

\medskip

\noindent
\textbf{Step 6.} \textit{Proving the claim $\alpha\cdot \gamma\cap \gamma\neq \emptyset$.}
Based on Step 5, we obtain an element $\phi\in \pstab(\rho)$. We claim that by taking $\alpha = \phi^d$ for some nonzero integer $d$, we have $\alpha\cdot \gamma\cap \gamma\neq \emptyset$ as desired in the original claim in Step 2.

Recall from \eqref{l_k} that 
\[
l_k
\geq 2m\cdot |\phi_1'|\cdots |\phi_s'| \cdot\cL
= 2\beta\cL
>  \beta\cL.
\]
Thus, $n_kl_k>n_k\beta\cL$.
By Remark \ref{remark: convex subcomplex C already gives essential hyperplane}, there is an essential hyperplane $H$ in $L_k$ crossing $\hgamma_b$, whose carrier $N(H)$ intersects $C_k$. Its orbit under the action of $\la \phi = \phi_k^{n_k\beta}\ra$ then gives at least $\cL+1$ intersections with $\hgamma_b$ given by
\[
v_0 = H \cap \hgamma_b,\quad 
v_1 = \phi(H) \cap \hgamma_b,\quad \dots,\quad
v_{\cL} = \phi^{\cL}(H) \cap \hgamma_b,
\]
which induce $\cL+1$ vertices of $H$ given below
\[
v_0,\quad \phi^{-1}(v_1),\quad\dots,\quad \phi^{-\cL}(v_{\cL}).
\]
By our construction of $\cL$ in Step 1\ref{item: possibilites to cross a hyperplane}, there exist $p\neq q$ such that $\phi^{-p}(v_p) = \phi^{-q}(v_q)$. Therefore, 
\[
v_q = \phi^{p-q}(v_p)\in \hgamma_b\cap \phi^{p-q}(\hgamma_b)
\subset \gamma\cap \phi^{p-q}(\gamma).
\]
This shows $\gamma\cap \alpha\cdot \gamma\neq \emptyset$ with $\alpha = \phi^{p-q}$. As mentioned in Step 2, since the combinatorial geodesic $\gamma$ is arbitrary, we conclude that \eqref{equation: upper bounda of diameter in main thm} holds. This completes the proof.
\end{proof}

\subsection{Almost cyclonormal paths}
In the proof of Theorem~\ref{theorem: main theorem}, we use the condition of almost cyclonormal edges to construction quasilines. In fact, this condition can be generalized by assuming almost cyclonormality for length-$n$ path stabilizers. 
\begin{defn}
Let $X_\Gamma$ be a graph of nonpositively curved cube complexes with the associated Bass--Serre tree $\cT$. Let $n>0$ be an integer. We say that $X_\Gamma$ has \textit{almost cyclonormal paths of length $n$} if for any vertex $\tv$ of $\cT$, the following conditions hold:
\begin{enumerate}
    \item For any path $\rho: [0,n]\to \cT$ of length $n$ starting from $\tv$. If $g_1,g_2,g_3\in \stab(\tv)$ represent distinct left cosets of $\pstab(\rho)$ in $\stab(\tv)$, then the subgroup $\pstab(\rho)^{g_1}\cap \pstab(\rho)^{g_2}\cap \pstab(\rho)^{g_3}$ is either trivial or cyclic.
    \item For any pair of distinct paths $\rho, \sigma: [0,n]\to \cT$ of length $n$ starting from $\tv$. If $g_1,g_2\in \stab(\tv)$ represent distinct left cosets of $\pstab(\sigma)$ in $\stab(\tv)$, then the subgroup $\pstab(\rho)\cap \pstab(\sigma)^{g_1}\cap \pstab(\sigma)^{g_2}$ is either trivial or cyclic.
\end{enumerate}
\end{defn}
Notice that $X_\Gamma$ has almost cyclonormal edges when it has almost cyclonormal paths of length $1$.
\begin{cor}\label{corollary: general version of main thm}
Let $X$ be a compact cube complex that splits as a graph of virtually special cube complexes $X_\Gamma$. If $X_\Gamma$ has almost cyclonormal paths of length $n$ for some integer $n>0$, then $\pi_1X$ has finite stature with respect to the vertex groups.
\end{cor}
Before proving Corollary \ref{corollary: general version of main thm}, we introduce the following results as analogues of Corollary \ref{corollary: gates are bounded}.

\begin{lem}\label{lem: bounded diameter for 2}
Let $X$ be a compact virtually special cube complex. For any constant $S>0$, there exists a constant $T$ depending on $X,S$ such that the following holds:

For any integer $m>0$ and a sequence of local isometries $A_1,\cdots, A_m\to X$ of compact virtually special cube complexes. Let $\tA_i$ be the based elevation of $A_i$ to $\tX$. If $\diam(A_i)\leq S$ for each $i$, then 
\[
\diam\left(\pf_{i=1}^mg_i\tA_i / \cap_{i=1}^m\stab(g_i\tA_i)\right)\leq T
\]
for any $g_1,\cdots, g_m\in \pi_1X$.
\end{lem}
\begin{proof}
Note that there are finitely many isomorphism classes of nonpositively curved cube complexes of diameter $\leq S$ that admit local isometries to $X$.
These yield finitely many possible local isometries to $X$, say $B_1,\cdots, B_{n}\to X$. Here, the complexes $B_1,\cdots, B_{n}$ may not be distinct, since they may correspond to different local isometries of the same isomorphism class. Then each local isometries $A_i\to X$ coincides with $B_{n_i}\to X$ for some $n_i$. 
As a result, by Corollary~\ref{corollary: gates are bounded}(1), there exists a constant $T$ depending on local isometries $B_1,\cdots, B_n\to X$, and hence depending on $X,S$, such that 
\[
\diam\left( \pf_{i=1}^mg_i \tA_{i}/\cap_{i=1}^m \pi_1A_{i}^{g_i^{-1}} \right)
= \diam\left( \pf_{i=1}^mg_i \tB_{n_i}/\cap_{i=1}^m \pi_1B_{n_i}^{g_i^{-1}} \right)
\leq T
\]
for any integer $m$ and elements $g_1,\cdots, g_m\in \pi_1X$.
\end{proof}

Recall that in the universal cover $\tX$, a finite path $\rho: [0,m]\to \cT$ corresponds to a subspace
\[
\bigsqcup_{i=0}^m \tX_i \sqcup \left(\bigsqcup_{j=1}^m\tY_j\times[0,1]\right)/\sim 
\]
where $X_0, X_1,\dots, X_m$ and $Y_1,\dots, Y_m$ are (not necessary distinct) vertex spaces and edge spaces in $X_\Gamma$.
\begin{cor}\label{lemma: bounded diameter for n}
Let $X$ be a compact cube complex that splits as a graph of virtually special cube complexes $X_\Gamma$. Let $\cT$ be the associated Bass--Serre tree. Then for any integer $n\geq 1$, there exists a constant $S$ depending on $n$ and $X$, such that for any path $\rho: [0,m]\to \cT$ of length $1\leq m\leq n$, with the notation above, we have 
\[
\diam \left(\tY_1^0\pf \tY_m^0/\pstab(\rho) \right) \leq S.
\]
\end{cor}

\begin{proof}
We prove this by induction on $n$. When $n = 1$, consider edge spaces $\{X_e\mid e\in \Gamma\}$ of $X_Gamma$. Take
\[
S_1 = \max\{\diam(X_e)\mid e\in E(\Gamma)\}.
\]
Since $\tY_1/\stab(\rho[0,1]) = Y_1$ is an edge space, its diameter is bounded by $S_1$.

Now suppose the statement holds for $n-1$ with a constant $S_{n-1}$. In particular, by our construction of $S_1$, we have $S_{n-1}\geq S_1$. For each vertex space $X_v$, by applying Lemma~\ref{lem: bounded diameter for 2} to the constant $S_{n-1}$, we obtain a constant $T_{v}$ depending on $X_v$ and $S_{n-1}$. We claim that $S = \max\{T_{v}, S_{n-1}\mid v\in V(\Gamma)\}$ satisfies the statement.

Indeed, for any path $\rho: [0,m]\to \cT$, if $m\leq n-1$, then the induction hypothesis gives 
\[
\diam \left(\tY_1^0\pf \tY_m^0/\pstab(\rho) \right) \leq S_{n-1}\leq S.
\]
If $\rho$ has length $n$, then we can decompose it as the concatenation $\rho = \rho[0,1]*\rho[1,n]$ of a length-$1$ path $\rho[0,1]$ and a length-$(n-1)$ path $\rho[1,n]$. In this case, we have 
\begin{align*}
\diam \left(\tY_1^0\pf \tY_m^0/\pstab(\rho) \right)
&= \diam \left(\tY_1^0\pf \left(\tY_2^0\pf \tY_m^0\right)/\pstab(\rho[0,1])\cap \pstab(\rho[1,n]) \right)\\
&= \diam \left(\tY_1^0\pf \left(\tY_2^0\pf \tY_m^0\right)/\stab\left(\tY_1^0\right)\cap \stab\left(\tY_2^0\pf \tY_m^0\right) \right).
\end{align*}
Note that $X_1$ is a vertex space $X_v$ for some $v\in V(\Gamma)$.
Because $\tY_1^0$ is an elevation of $\tY_1^0/\stab\left(\tY_1^0\right)$ to $\tX_v$, and $\tY_2^0\pf \tY_m^0$ is an elevation of $\left(\tY_2^0\pf \tY_m^0\right)/\stab\left(\tY_2^0\pf \tY_m^0\right)$. By induction hypothesis, we have 
\[
\diam\left(Y_1\right)\leq S_1\leq S_{n-1},\quad
\diam\left(\tY_2^0\pf \tY_m^0\right)/\stab\left(\tY_2^0\pf \tY_m^0\right)\leq S_{n-1}.
\]
Thus, by Lemma~\ref{lem: bounded diameter for 2} and our construction of $T_v$, we have 
\[
\diam \left(\tY_1^0\pf \left(\tY_2^0\pf \tY_m^0\right)/\stab\left(\tY_1^0\right)\cap \stab\left(\tY_2^0\pf \tY_m^0\right) \right)\leq T_v\leq S.
\]
This finishes the induction and hence completes the proof.
\end{proof}

\begin{proof}[Proof of Corollary \ref{corollary: general version of main thm}]
Suppose that the graph of nonpositively curved cube complexes $X_\Gamma$ has almost cyclonormal paths of length $n$.
As mentioned in the proof of Theorem \ref{theorem: main theorem}, it suffices to show that for any path $\rho: [0,s]\to \cT$ in the associated Bass--Serre tree, we have 
\begin{equation}\label{equality: bound in general case}
\diam\left( \tY_1^0\pf \tY_s^0/\pstab(\rho) \right)
\leq P
\end{equation}
for some constant $P$ depending only on $X$. 
To prove \eqref{equality: bound in general case}, we first make the following observation of paths on $\cT$.

Let $\sigma: [0,m]\to \cT$ be a path of length $m\leq n$ starting from $\tv = \sigma(0)$, where $\pstab(\tv) = \pi_1X_v$. Suppose that it corresponds to the subspace $\bigsqcup_{i=0}^n \tX_i \sqcup \left(\bigsqcup_{j=1}^n\tY_j\times[0,1]\right)/\sim$ in the universal cover $\tX$. Denote by $\tY_\sigma$ the projection $\tY_1^0\pf \tY_n^0$. Then the pointwise stabilizer $\pstab(\sigma)$ is represented by a local isometry $\tY_\sigma/\pstab(\sigma)\to X_0 = X_v$. Similarly, let $\xi:[0,r]\to \cT$ be another (not necessarily distinct) path of length $r\leq n$ starting from $\tv = \xi(0)$. Then $\pstab(\xi)$ is represented by a local isometry $\tY_\xi/\pstab(\xi) \to X_v$.

Since $\sigma,\xi$ are both paths of length $\leq n$. By Corollary~\ref{lemma: bounded diameter for n}, there exists a constant $S$ such that
\begin{equation}\label{bound for path}
\diam\left(\tY_\sigma/\pstab(\sigma)\right)\leq S,\quad 
\diam\left(\tY_\xi/\pstab(\xi)\right)\leq S.
\end{equation}
By applying Lemma~\ref{lem: bounded diameter for 2} to the constant $S$ and the vertex space $X_v$, we obtain a constant $T_{v}$. In particular, for any $g\in \pi_1X_v$, we have
\begin{equation}\label{bound for triple}
\diam\left(\tY_{\sigma}\pf \tY_{\xi} \pf g\tY_{\xi}/\pstab(\sigma)\cap \pstab(\xi)\cap \pstab(\xi)^{g^{-1}}\right)\leq T_v.
\end{equation}
So there are only finitely many isomorphism classes of such complexes. 

Suppose that $\sigma,\xi$ have length $n$.
When there exists $g\in \pstab(\sigma)\setminus \pstab(\xi)$, the three subgroups $\pstab(\sigma), \pstab(\xi), \pstab(\xi)^{g^{-1}}$ are distinct in $\pstab(\tv)$. Because the $X_\Gamma$ has almost cyclonormal paths of length $n$, the intersection $\pstab(\sigma)\cap \pstab(\xi)\cap \pstab(\xi)^{g^{-1}}$ is either trivial or infinite cyclic. When it is cyclic, by taking a generator $\phi$, we obtain a quasiline
\begin{equation}\label{quasiline n length type}
\left( \tY_{\sigma}\pf \tY_{\xi} \pf g\tY_{\xi}, \phi \right).
\end{equation}
Since there are finitely many isomorphism classes of complexes of the form in \eqref{bound for triple}, there are finitely many isomorphism classes of such quasilines.

Now we apply the same argument as in the proof of Theorem \ref{theorem: main theorem}, and replace the quasiline representatives in Step 1\ref{item: representative of quasilines and trivial hyperplane} with the finitely many isomorphism classes of quasilines of the form in \eqref{quasiline n length type}. The same proof then yields a constant $P_n$ such that 
\[
\diam\left( \tY_1^0\pf \tY_{kn}^0/\pstab(\rho) \right) \leq P_n
\]
for any integer $k\geq 0$. This proves \eqref{equality: bound in general case} when the length of $\rho$ is $kn$.

More generally, for a path $\rho:[0,s]\to \cT$ of length $s$, suppose $s = kn+m$ for some integers $k\geq 0$ and $0\leq m<n$. Then we can decompose $\rho$ as the concatenation $\rho = \rho[0,kn]* \rho[kn, s]$. Recall that it corresponds to the subspace $\bigsqcup_{i=0}^s \tX_i \sqcup \left(\bigsqcup_{j=1}^s\tY_j\times[0,1]\right)/\sim$ in $\tX$. Because the length of $\rho[0,kn]$ is a multiple of $n$, the above argument shows that 
\[
\diam\left( \tY_1^0\pf \tY_{kn}^0/\pstab(\rho[0,kn]) \right)
= \diam\left( \tY_{kn}^1\pf \tY_{1}^1/\pstab(\rho[0,kn]) \right)
\leq P_n.
\]
Since $m<n$, Corollary~\ref{lemma: bounded diameter for n} and the discussion in \eqref{bound for path} imply that 
\[
\diam\left( \tY_{kn+1}^0\pf \tY_{s}^0/\pstab(\rho[kn,s]) \right)\leq S
\]
Suppose $X_{kn} = X_u$ for some vertex $u\in V(\Gamma)$.
By applying Lemma~\ref{lem: bounded diameter for 2} to the constant $\max\{P_n,S\}$ and local isometries
\[
\left( \tY_{kn}^1\pf \tY_{1}^1/\pstab(\rho[0,kn]) \right)\to X_{kn} = X_u,\quad
\tY_{kn+1}^0\pf \tY_{m}^0/\pstab(\rho[kn,s]) \to X_u,
\]
we obtain a constant $P_u$ such that
\begin{align*}
\diam\left( \tY_1^0\pf \tY_s^0/\pstab(\rho) \right)
&= \diam\left( \left(\tY_1^0\pf \tY_{kn}^0\right)\pf  \left(\tY_{kn+1}^0\pf \tY_{m}^0\right)/\pstab(\rho) \right)\\
&= \diam\left( \left(\tY_{kn}^1\pf \tY_{1}^1\right)\pf  \left(\tY_{kn+1}^0\pf \tY_{m}^0\right)/\pstab(\rho)\right)\\
&\leq P_u.
\end{align*}
Therefore, we obtain \eqref{equality: bound in general case} by taking the maximum $P = \max\{P_u\mid u\in V(\Gamma)\}$. This finishes the proof.
\end{proof}

\bibliographystyle{alpha}	
\bibliography{ref}

\begin{thebibliography}{GMRS98}

\bibitem[AGM16]{AgolGrovesManning2016}
Ian Agol, Daniel Groves, and Jason~Fox Manning.
\newblock An alternate proof of wise's malnormal special quotient theorem.
\newblock {\em Forum of Mathematics, Pi}, 4:e1, 2016.

\bibitem[Ago13]{Agol2013VirtualHaken}
Ian Agol.
\newblock The virtual haken conjecture.
\newblock {\em Documenta Mathematica}, 18:1045--1087, 2013.
\newblock With an appendix by Ian Agol, Daniel Groves, and Jason Manning.

\bibitem[BH99]{BridsonHaefliger1999}
Martin~R. Bridson and Andr{\'e} Haefliger.
\newblock {\em Metric Spaces of Non-Positive Curvature}, volume 319 of {\em Grundlehren der mathematischen Wissenschaften}.
\newblock Springer-Verlag Berlin Heidelberg, Berlin, Heidelberg, 1st edition, 1999.

\bibitem[Bow24]{Bowditch2024Median}
Brian~H. Bowditch.
\newblock Median algebras.
\newblock \url{https://bhbowditch.com/papers/median-algebras.pdf}, 2024.

\bibitem[CFI16]{ChatterjiFernosIozzi2016}
Indira Chatterji, Talia Fern{\'o}s, and Alessandra Iozzi.
\newblock The median class and superrigidity of actions on {CAT(0)} cube complexes.
\newblock {\em Journal of Topology}, 9(2):349--400, 2016.

\bibitem[CS11]{CapraceSageev2011RankRigidity}
Pierre{-}Emmanuel Caprace and Michah Sageev.
\newblock Rank rigidity for {CAT(0)} cube complexes.
\newblock {\em Geometric and Functional Analysis}, 21(4):851--891, 2011.

\bibitem[DK92]{DuchampKrob1992}
G.~Duchamp and D.~Krob.
\newblock The lower central series of the free partially commutative group.
\newblock {\em Semigroup Forum}, 45:385--394, 1992.

\bibitem[Fio23]{Fioravanti2023}
Elia Fioravanti.
\newblock On automorphisms and splittings of special groups.
\newblock {\em Compositio Mathematica}, 159(2):232--305, 2023.

\bibitem[GM22]{GrovesManning2022}
Daniel Groves and Jason~Fox Manning.
\newblock Specializing cubulated relatively hyperbolic groups.
\newblock {\em Journal of Topology}, 15(2):398--442, 2022.

\bibitem[GMRS98]{GitikMitraRipsSageev1998}
Rita Gitik, Mahan Mitra, Eliyahu Rips, and Michah Sageev.
\newblock Widths of subgroups.
\newblock {\em Transactions of the American Mathematical Society}, 350(1):321--329, 1998.

\bibitem[Hag19]{Hagen2019}
Mark~F. Hagen.
\newblock Lecture notes on {CAT(0)} cube complexes, median graphs, and cubulating groups.
\newblock \url{https://www.wescac.net/into_the_forest.pdf}, 2019.

\bibitem[Hag23]{Haglund2023}
Fr{\'e}d{\'e}ric Haglund.
\newblock Isometries of {CAT(0)} cube complexes are semi-simple.
\newblock {\em Annales math{\'e}matiques du Qu{\'e}bec}, 47:249--261, 2023.

\bibitem[HJP16]{HuangJankiewiczPrzytycki2016Cocompactly}
Jingyin Huang, Kasia Jankiewicz, and Piotr Przytycki.
\newblock Cocompactly cubulated 2-dimensional artin groups.
\newblock {\em Commentarii Mathematici Helvetici}, 91(3):519--542, 2016.

\bibitem[Hua18]{Huang2018Commensurability}
Jingyin Huang.
\newblock Commensurability of groups quasi-isometric to raags.
\newblock {\em Inventiones Mathematicae}, 213(3):1179--1247, September 2018.

\bibitem[HW08]{Haglund-Wise2008}
Fr{\'e}d{\'e}ric Haglund and Daniel~T. Wise.
\newblock Special cube complexes.
\newblock {\em Geometric and Functional Analysis (GAFA)}, 17(5):1551--1620, 2008.

\bibitem[HW12]{haglund-wise2012}
Frédéric Haglund and Daniel~T. Wise.
\newblock A combination theorem for special cube complexes.
\newblock {\em Annals of Mathematics}, 176(3):1427--1482, 2012.

\bibitem[HW19]{Huang--Wise2019statureseparabilitygraphsgroups}
Jingyin Huang and Daniel~T. Wise.
\newblock Stature and separability in graphs of groups, 2019.

\bibitem[HW24]{Huang--Wise2024fintiestature}
Jingyin Huang and Daniel~T. Wise.
\newblock Virtual specialness of certain graphs of special cube complexes.
\newblock {\em Mathematische Annalen}, 388:329--357, 2024.

\bibitem[Jan23]{jankiewicz2023finitestatureartingroups}
Kasia Jankiewicz.
\newblock Finite stature in artin groups, 2023.

\bibitem[Min12]{Minasyan2012}
Ashot Minasyan.
\newblock Hereditary conjugacy separability of right-angled artin groups and its applications.
\newblock {\em Groups, Geometry, and Dynamics}, 6(2):335--388, 2012.

\bibitem[OR23]{OregonReyes2023}
Eduardo Oreg{\'o}n-Reyes.
\newblock On cubulated relatively hyperbolic groups.
\newblock {\em Geometry \& Topology}, 27(2):575--640, 2023.

\bibitem[Sag95]{Sageev1995}
Michah Sageev.
\newblock Ends of group pairs and non‐positively curved cube complexes.
\newblock {\em Proceedings of The London Mathematical Society}, pages 585--617, 1995.

\bibitem[Ser80]{Serre1980}
Jean-Pierre Serre.
\newblock {\em Trees}.
\newblock Springer-Verlag, Berlin, Heidelberg, 1 edition, 1980.

\bibitem[She23]{Shepherd2023}
Sam Shepherd.
\newblock Imitator homomorphisms for special cube complexes.
\newblock {\em Transactions of the American Mathematical Society}, 376:599--641, 2023.

\bibitem[She24]{Shephard2024rightangledbuilding}
Sam Shepherd.
\newblock Commensurability of lattices in right-angled buildings.
\newblock {\em Advances in Mathematics}, 441:109522, 2024.

\bibitem[She25]{shepherd2025productseparabilityspecialcube}
Sam Shepherd.
\newblock Product separability for special cube complexes, 2025.

\bibitem[SW79]{ScottWall1979}
Peter Scott and Terry Wall.
\newblock Topological methods in group theory.
\newblock In C.~T.~C. Wall, editor, {\em Homological Group Theory}, number~36 in London Mathematical Society Lecture Note Series, pages 137--203. Cambridge University Press, Cambridge, UK, 1979.

\bibitem[Wil08]{Wilton2008}
Henry Wilton.
\newblock Hall's theorem for limit groups.
\newblock {\em Geometric and Functional Analysis (GAFA)}, 18(1):271--303, 2008.

\bibitem[Wis96]{Wise1996NPCSquareComplexes}
Daniel~T. Wise.
\newblock {\em Non-positively curved squared complexes: Aperiodic tilings and non-residually finite groups}.
\newblock Ph.d. thesis, Princeton University, 1996.

\bibitem[Wis04]{Wise2004Sectional}
Daniel~T. Wise.
\newblock Sectional curvature, compact cores, and local quasiconvexity.
\newblock {\em Geometric and Functional Analysis}, 14(2):433--468, 2004.

\bibitem[Wis21]{wise2021}
Daniel~T. Wise.
\newblock {\em The Structure of Groups with a Quasiconvex Hierarchy: (AMS-209)}, volume 366.
\newblock Princeton University Press, 2021.

\end{thebibliography}

\end{document}